\documentclass{article}

\usepackage[utf8]{inputenc}
\usepackage[italian,english]{babel}

\usepackage{geometry}								
\usepackage{indentfirst}  		

\usepackage[linktoc=all]{hyperref}	
\usepackage{tikz}				

\usepackage{amsmath}
\usepackage{amssymb}	
\usepackage{mathrsfs}	
\usepackage{amsthm}	

\theoremstyle{plain}
\newtheorem{mythm}{Theorem}[section]
\newtheorem{myprop}[mythm]{Proposition}
\newtheorem{mylemma}[mythm]{Lemma}

\newtheorem{mydef}[mythm]{Definition}

\newtheorem{introthm}{Theorem}

\theoremstyle{remark}

\newtheorem*{myrmk}{Remark}

\newcommand{\ol}[1]{\overline{#1}}

\newcommand{\oc}[1]{\widehat{#1}}

\newcommand{\Rar}{\Rightarrow}
\newcommand{\rar}{\rightarrow}

\newcommand{\gen}[1]{\langle#1\rangle}
\newcommand{\ggen}[1]{\langle\langle#1\rangle\rangle}
\newcommand{\abs}[1]{|#1|}

\newcommand{\bZ}{\mathbb Z}
\newcommand{\bN}{\mathbb N}

\newcommand{\fI}{\mathfrak I}
\newcommand{\fJ}{\mathfrak I} 
\newcommand{\cA}{\mathcal A}
\newcommand{\cB}{\mathcal B}

\newcommand{\cH}{\mathcal H}
\newcommand{\cL}{\mathcal L}
\newcommand{\cM}{\mathcal M}
\newcommand{\cN}{\mathcal N}
\newcommand{\cP}{\mathcal P}
\newcommand{\cQ}{\mathcal Q}
\newcommand{\cR}{\mathcal R}

\newcommand{\sgr}{\le}
\newcommand{\nor}{\trianglelefteq}
\newcommand{\rank}[1]{\text{rank}(#1)}

\newcommand{\core}[1]{\text{core}(#1)}
\newcommand{\bcore}[1]{\text{core}_*(#1)}

\newcommand{\ex}{\text{exp}}
\newcommand{\pol}{\text{pol}}

\newcommand{\rlan}[1]{\cL_{\text{r}}(#1)}
\newcommand{\clan}[1]{\cL_{\text{cr}}(#1)}
\newcommand{\norma}[1]{\|#1\|}
\newcommand{\ram}[1]{\mathrm{ram}(#1)}
\newcommand{\pro}{\rightarrow}
\newcommand{\der}{\xrightarrow{*}}

\newcommand{\Psecidealfingen}{\text{Section 3}}

\newcommand{\Pdegreeset}{\text{Theorem C}}

\newcommand{\Pexamplecyclic}{\text{Section 6.1}}
\newcommand{\Pexampleeven}{\text{Section 6.2}}
\newcommand{\Pexampleodd}{\text{Section 6.3}}

\title{Ideals of equations for elements in a free group and context-free languages}
\author{
Dario Ascari \thanks{The author was funded by the Engineering and Physical Sciences Research Council.}\\
{\small \textit{Mathematical Institute, Andrew Wiles Building,}}\\
{\small \textit{University of Oxford, Oxford OX2 6GG, UK}}\\
{\small e-mail: \texttt{ascari@maths.ox.ac.uk}}\\
\\
}

\begin{document}

\maketitle

\begin{abstract}
Let $F$ be a finitely generated free group, and $H\sgr F$ a finitely generated subgroup. An equation for an element $g\in F$ with coefficients in $H$ is an element $w(x)\in H*\gen{x}$ such that $w(g)=1$ in $F$; the degree of the equation is the number of occurrences of $x$ and $x^{-1}$ in the cyclic reduction of $w(x)$. Given an element $g\in F$, we consider the ideal $\fI_g\subseteq H*\gen{x}$ of equations for $g$ with coefficients in $H$; we study the structure of $\fI_g$ using context-free languages.

We describe a new algorithm that determines whether $\fI_g$ is trivial or not; the algorithm runs in polynomial time. We also describe a polynomial-time algorithm that, given $d\in\bN$, decides whether or not the subset $\fJ_{g,d}\subseteq\fI_g$ of all degree-$d$ equations is empty. We provide a polynomial-time algorithm that computes the minimum degree $d_{\min}$ of a non-trivial equation in $\fI_g$. We provide a sharp upper bound on $d_{\min}$.

Finally, we study the growth of the number of (cyclically reduced) equations in $\fI_g$ and in $\fJ_{g,d}$ as a function of their length. We prove that this growth is either polynomial or exponential, and we provide a polynomial-time algorithm that computes the type of growth (including the degree of the growth if it's polynomial).
\end{abstract}

\begin{center}
\end{center}


\section{Introduction}

The study of polynomial equations over fields and rings has played a huge role in mathematics along the centuries. In analogy with standard algebraic geometry, a theory of equations in the non-commutative setting of group theory was developed, see \cite{BMR99}, \cite{KM98a}, \cite{KM98b}. Given two groups $H\sgr G$, an \textit{equation} in the variables $x_1,...,x_m$ with coefficients in $H$ is an equality of the form $w(x_1,...,x_m)=1$ for some element $w(x_1,...,x_m)\in H*F(x_1,...,x_m)$, where $F(x_1,...,x_m)$ denotes the free group with basis $x_1,...,x_m$; an $m$-tuple of elements $(g_1,...,g_m)$ of elements of $G$ is a \textit{solution} to the equation if $w(g_1,...,g_m)=1$ in $G$.

The problem of solving (systems of) equations over groups has been extensively studied in literature, and turned out to be much harder than its polynomial counterpart. In the case of two finitely generated free groups $H\sgr F$, the first algorithm to solve equations (i.e. to find solutions in $F$ to equations with coefficients in $H$) was found by G. S. Makanin \cite{Mak83}; this produced a rich and fruitful theory, leading to the introduction of \textit{limit groups} and to the construction of \textit{Makanin-Razborov diagrams} (see \cite{Raz85}, \cite{Sel01}), and providing strong results about the first order theory of free groups (see \cite{Mak85}, and also \cite{Sel06} and \cite{KM06} for the solution to Tarski's problem about free groups). In general the set of solutions to a system of equations is extremely hard to describe; for the particular case of equations in one variable, a description of the set of solutions can be found in \cite{BGM06}.

More recently, the study of the dual problem was initiated, opening a new perspective into the topic: given two finitely generated free groups $H\sgr F$ and an element $g\in F$, we are interested in understanding the \textbf{ideal} $\fI_g=\{w(x)\in H*\gen{x} : w(g)=1$ in $F\}$. The ideal $\fI_g\subseteq H*\gen{x}$ is a normal subgroup, namely the kernel of the evaluation homomorphism $\varphi_g:H*\gen{x}\rar F$ that sends $x$ to $g$.

In \cite{RV23}, A. Rosenmann and E. Ventura provide an algorithm that, given $H$ and $g$, determines whether the ideal $\fI_g$ is trivial or not. Following their terminology, we say that $g\in F$ \textbf{depends} on $H\sgr F$ if $\fI_g\not=1$ or, equivalently, $\rank{\gen{H,g}}\le\rank{H}$. In \cite{RV23} they prove that the set of elements of $F$ that depend on $H$ is a finite union of double cosets $Hg_1H,...,Hg_kH$, and that the elements $g_1,...,g_k$ can be algorithmically computed from $H$ and $g$.

In a previous paper \cite{PART1}, the author studied the structure of the ideal $\fI_g$ using the technique of folding of graphs, paying particular attention to the degree of the equations. For an equation $w(x)\in\fI_g$ define its \textbf{degree} as the number of occurrences of $x$ and $x^{-1}$ in its cyclic reduction. We produced an algorithm that computes the minimum possible degree for a non-trivial equation in $\fI_g$, and we showed that all the equations of a certain degree $d$ can be characterized by looking only at the short equations of the same degree $d$ (i.e. at the equations up to a certain length, which depends on $H,g,d$).

In the current paper we take a different approach, studying the structure of the ideal $\fI_g$ with techniques based on context-free languages. Context-free languages have been widely studied in computer science, and they arise naturally in group theory. In particular, a celebrated theorem of D. E. Muller and P. E. Schupp states that the word problem of a finitely generated group is context-free if and only if the group is virtually free; for more details see \cite{MS83}. Context-free languages allow us to improve some of the results of our previous paper \cite{PART1}, in particular those concerning the time-complexity of the algorithms. Context-free languages also allow us to obtain information about the growth of the number of equations in $\fI_g$ of a certain degree $d$ as a function of their length.

We point out that, besides context-free languages, there are also other families of languages (and automata) that have been fruitfully used in
the study of problems related to group theory; examples of such families are \textit{indexed languages} (see \cite{BG96}, \cite{Bri05}, \cite{HR06}) and ET$0$L
\textit{and} EDT$0$L \textit{languages} (see \cite{CEF17}). For applications of EDT$0$L languages to the study of equations over groups we refer the reader to \cite{CDE15} and \cite{DJK16}. Families of formal languages also arise naturally when studying \textit{combings} in a group, i.e. normal forms for the elements of the group (see \cite{BG93} and \cite{Bri05}).

\subsection*{Results and structure of the paper}

Let $F_n$ be a finitely generated free group of rank $n\in\bN$; fix a finitely generated subgroup $H\sgr F_n$ and an element $g\in F_n$. In this paper we study the structure of the ideal $\fI_g=\{w(x)\in H*\gen{x} : w(g)=1$ in $F_n\}\nor H*\gen{x}$. Most of the results can be generalized to equations in more variables but, for simplicity of notation, for most of the paper we focus on the one-variable case; the results in more variables can be found in Section \ref{SectionMultivariate} at the end of the paper.

\

In Section \ref{SectionCfree} we briefly recall the notion of context-free grammar and language; for more details the reader can refer to \cite{HMU06}. We restate some classical result about context-free languages. We also recall the notion of context-free subset of a free group, which will be fundamental in the whole paper.

\

In Section \ref{SectionCfreeIdeals} we prove that the ideal $\fI_g$ is a context-free subset of $H*\gen{x}$, and we provide a new algorithm (substantially different from the ones already in literature) that determines whether the ideal is non-trivial; we will prove later, in Section \ref{SectionAlgorithms}, that the algorithm runs in polynomial-time.

\begin{introthm}[See Theorem \ref{cfreeIg}]\label{introcfreeIg}
The set $\fI_g$ is context-free as a subset of $H*\gen{x}$.
\end{introthm}

\begin{introthm}[See Theorems \ref{cfreeIg2} and \ref{cfreeIg3}]\label{introcfreeIg2}
There is a polynomial-time algorithm that tells us whether $\fI_g$ contains a non-trivial equation or not and, in case it does, produces a non-trivial equation in $\fI_g$.
\end{introthm}

Our algorithm consists essentially of building a context-free grammar for the context-free language of the ideal $\fI_g$, and then checking whether the language generated by this grammar is non-empty.

\begin{myrmk}
We point out that two other algorithms to determine whether $\fI_g$ is non-trivial were already known, one based on Nielsen transformations (see \cite{RV23}) and the other based on Stallings' folding operations (see both \cite{PART1} and \cite{RV23}). These two algorithms are very different in nature from the context-free language algorithm presented in this paper. The Stallings' folding algorithm provides more accurate information about the ideal in a group theoretic sense, as it allows to explicitly compute a finite set of generators of the ideal $\fI_g$ as normal subgroup of $H*\gen{x}$. However, the context-free language algorithm allows us for a better description of the set of all words representing elements of $\fI_g$; in particular, it makes it easier (and faster in terms of running-time) to determine whether the ideal contains an equation of a given degree $d$, and it allows us to study in detail the growth of the number of equations as a function of their length.
\end{myrmk}

We then turn our attention to the set $\fJ_{g,d}$ of the equations of a certain degree $d$, and we prove similar results. For an integer $d\ge1$, define $\fJ_{g,d}$ to be the subset of $\fI_g$ consisting of the equations of degree $d$.

\begin{introthm}[See Theorem \ref{cfreeJgd}]\label{introcfreeJgd}
The set $\fJ_{g,d}$ is context-free as a subset of $H*\gen{x}$.
\end{introthm}

\begin{introthm}[See Theorems \ref{cfreeJgd2} and \ref{cfreeJgd3}]\label{introcfreeJgd2}
There is a polynomial-time algorithm that tells us whether the set $\fJ_{g,d}$ is non-empty and, if so, produces an element in $\fJ_{g,d}$.
\end{introthm}

\

In Section \ref{SectionAsymptotic} we study the growth of the number of equations in the sets $\fI_g$ and $\fJ_{g,d}$. Define $\rho_g(M)$ to be the number of cyclically reduced words of length at most $M$ that represent an equation in $\fI_g$. Similarly, define $\rho_{g,d}(M)$ to be the number of cyclically reduced words of length at most $M$ that represent an equation in $\fJ_{g,d}$. Then we prove the following results:

\begin{introthm}[See Theorem \ref{growth1}]\label{introgrowth1}
The function $\rho_g(M)$ has exponential growth.
\end{introthm}

\begin{introthm}[See Theorems \ref{growth2} and \ref{growth3}]\label{introgrowth2}
Let $d\in\bN$ be a non-negative integer. Then the function $\rho_{g,d}(M)$ either has exponential growth or else it is bounded above and below by polynomials of degree $k$ for some $k\in\bN$. Moreover, there is a polynomial-time algorithm that tells us which case takes place, and in the second case computes the integer $k$.
\end{introthm}

In order to prove Theorems \ref{introgrowth1} and \ref{introgrowth2}, we rely on the fact that the sets $\fI_g$ and $\fJ_{g,d}$ are context-free. The growth rate of arbitrary context-free languages has already been studied, see \cite{BG02} and \cite{Inc01} and \cite{GKRS08}.

We then study the set $D_g\subseteq\bN$ of possible degrees for non-trivial equations in $\fI_g$. In \cite{PART1} the author proved that $D_g$ is either equal to $\bN$ or to $2\bN$, up to a finite set; more precisely, if $D_g$ contains an odd number then $\bN\setminus D_g$ is finite, while if $D_g$ only contains even numbers then $2\bN\setminus D_g$ is finite. Moreover he established an algorithm that determines which case takes place, and computes the finite set $\bN\setminus D_g$ or $2\bN\setminus D_g$ respectively. In this paper we consider the partition of $D_g$ based on the growth rate of the function $\rho_{g,d}$:
$$D_g=D^\ex_g\sqcup\bigsqcup_{k\in\bN}D^{\pol,k}_g$$
where $D_g^{\pol,k}:=\{d\in D_g : \rho_{g,d}$ has polynomial growth of degree $k\}$ and $D_g^\ex=\{d\in D_g : \rho_{g,d}$ has exponential growth$\}$. We prove that this partition can be algorithmically computed too.

\begin{introthm}[See Theorem \ref{Dgpolk}]\label{introDgpolk}
Suppose $H$ has rank at least $2$. Then there is an algorithm that computes the finite sets $D^{\pol,k}_g$ for $k\in\bN$; all but finitely many of these are empty.
\end{introthm}

\

In Section \ref{SectionAlgorithms} we study the running time of the algorithms. We assume that we are working on a RAM machine; we assume that we can perform in time $O(1)$ the most basic operations, including sum and multiplication of two integer numbers, with this model of computation. According to \cite{HMU06}, given a context-free grammar, there is a linear-time algorithm to decide whether the corresponding language is non-empty; we describe such algorithm and some of its variants. We provide an unambiguous grammar for the language of an ideal $\fI_g$.

This allows us to estimate the running time of the algorithms of Theorems \ref{introcfreeIg2} and \ref{introcfreeJgd2}, providing polynomial bounds. Let $n$ be the number of generators of the ambient free group $F_n$, let $r$ be the rank of the subgroup $H$ and let $L$ be the maximum length of the words representing $h_1,...,h_r,g$, where $h_1,...,h_r$ is a fixed basis for $H$. Then we prove that the algorithm of Theorem \ref{introcfreeIg2} runs in time $O(nr^9L^3)$ and the one of Theorem \ref{introcfreeJgd2} in time $O(nr^{15}L^3d^6)$.

It is natural to wonder about the minimum possible degree $d_{\min}$ for a non-trivial equation in $\fI_g$. An algorithm for finding $d_{\min}$ had already been provided by the author in \cite{PART1}. Improving on this, we provide a polynomial-time such algorithm, very different in nature from the one of \cite{PART1}, and we give estimates on how big $d_{\min}$ can be.

\begin{introthm}[See Theorem \ref{dmin}]\label{introdmin}
There is a polynomial-time algorithm that computes the minimum possible degree for an equation in $\fI_g$.
\end{introthm}

The algorithm of Theorem \ref{introdmin} runs in time $O(nr^9L^3\log(nrL))$.

\begin{introthm}[See Theorem \ref{bounddmin}]\label{introbounddmin}
If it is non-trivial, the ideal $\fI_g$ contains a non-trivial equation of degree at most $2K_{r+1}L^{M_{r+1}}$.
\end{introthm}

In the above statement, the constants $K_{r+1},M_{r+1}$, which only depend on $r$, are taken from \cite{LSV15}; in particular, the above estimate is polynomial in $L$, but not in $r$. We show by means of an example that $d_{\min}$ can be exponential in $r$, so that no completely polynomial bound can be given on $d_{\min}$.
 
\

In Section \ref{SectionExamples} we give examples showing that several kinds of growth can take place for $\rho_{g,d}$. In Section \ref{examplecyclic} we deal with the case where $\rank{H}=1$, which was excluded from Theorem \ref{introDgpolk}. In Example \ref{example46} we have that $\rho_{g,4}(M)$ is constant for all $M$ big enough. Example \ref{exampledmin} shows that $d_{\min}$ can grow exponentially in the rank $r$ of the subgroup $H$, showing that the bound of Theorem \ref{introbounddmin} is sharp.

\

In Section \ref{SectionMultivariate} we show how several of the results can be generalized from the one-variable to the multivariate case.

\section*{Acknowledgements}

The author would like to thank his supervisor M. R. Bridson for the suggestion of looking at regular languages and automata, and for all the useful discussion while working on the present work.

\section{Context-free languages and free groups}\label{SectionCfree}

In this section we introduce the notions of regular language and of context-free language, which are widely used in computer science; see for example \cite{HMU06} for an exhaustive introduction to the topic. We also discuss the notion of context-free subset of a free group.

Let $\cA$ be a finite set, which we call the \textbf{alphabet}; the elements of $\cA$ will be called \textbf{letters}. The free monoid over $\cA$ will be denoted $\cA^*$; the elements of $\cA^*$ will be called \textbf{words}. The length of a word $w\in\cA^*$, denoted $l(w)$, is the number of letters in the word. A \textbf{language} is any subset $\cL\subseteq\cA^*$.

%
%

\subsection{Regular languages}\label{SectionRegularLanguages}

\begin{mydef}
A \textbf{finite automaton} is a quadruple $(Q,\delta,Q_0,Q_f)$ where:

(i) $Q$ is a finite set, whose elements are called \textbf{states} of the automaton.

(ii) $\delta:Q\times\cA\rar\wp(Q)$ is a function, called \textbf{transition function}.

(iii) $Q_0$ is a subset of $Q$, whose elements are called \textbf{initial states}.

(iv) $Q_f$ is a subset of $Q$, whose elements are called \textbf{final states}.
\end{mydef}

The transition function tells us that, if we are at a state $q\in Q$ and we read a letter $c\in\cA$, then we are allowed to move to any of the states in the set $\delta(q,c)\subseteq Q$. More generally, given a subset of states $T\subseteq Q$ and a word $w\in\cA^*$, define $\delta(T,w)$ to be the subset of states to which we can move by reading $w$ from a state in $T$; to be formal, if $w=w_1...w_{l(w)}$, then we define inductively $T_0=T$ and $T_{i+1}=\bigcup_{q\in T_i}\delta(q,w_{i+1})$, and we define $\delta(T,w)=T_{l(w)}$.

We say that the automaton \textbf{accepts} a word $w$ if the intersection of $\delta(Q_0,w)$ and $Q_f$ is nonempty; this means that there is at least one way to start at one of the initial states and to move around the automaton reading $w$ and ending up in a final state.

\begin{mydef}
A finite automaton $(Q,\delta,Q_0,Q_f)$ is called \textbf{deterministic} if it satisfies:

(i) $\delta(q,a)$ is a singleton for each $q\in Q$ and $a\in\cA$.

(ii) $Q_0$ is a singleton.

For a deterministic finite automaton we call $q_0$ the unique initial state, i.e. $Q_0=\{q_0\}$. With an abuse of notation, we also consider the transition function as a map $\delta:Q\times\cA\rar Q$.
\end{mydef}

Condition (ii) means that there is a unique initial state; condition (i) means that from each state and for each letter that we read, we can jump on exactly one state. We don't require any condition on the set $Q_f$.

\begin{myprop}
Let $\cL\subseteq\cA^*$ be a language. Then the following are equivalent:

(i) There is a finite automaton such that $w\in\cL$ if and only if the automaton accepts $w$.

(ii) There is a deterministic finite automaton such that $w\in\cL$ if and only if the automaton accepts $w$.

In this case $\cL$ is called \textbf{regular language}.
\end{myprop}
\begin{proof}
See Theorem 2.12 of \cite{HMU06}.
\end{proof}

\begin{myprop}
Let $\cL,\cL'$ be regular languages. Then we have the following:

(i) $\cL\cup\cL'$ is a regular language.

(ii) $\cL\cap\cL'$ is a regular language.

(iii) The complement $\cL^c$ of $\cL$ is a regular language.
\end{myprop}
\begin{proof}
For part (iii) see Theorem 4.5 of \cite{HMU06}. For part (ii) see Theorem 4.8 of \cite{HMU06}. For part (i) use that $\cL\cup\cL'=((\cL)^c\cap(\cL')^c)^c$.
\end{proof}

\subsection{Context-free grammars}

\begin{mydef}
A \textbf{context-free grammar} is a triple $(\cN,\cP,S)$ where

(i) $\cN$ is a finite set, disjoint from $\cA$, whose elements are called \textbf{non-terminal symbols}.

(ii) $\cP$ is a finite set of \textbf{production rules} $(N,u)$ with $N\in\cN$ and $u\in(\cA\cup\cN)^*$.

(iii) $S\in\cN$ is a fixed element, called the \textbf{initial symbol}.
\end{mydef}

We will denote a production rule $(N,u)\in\cP$ also as $N\pro u$. If there are several production rules $(N,u),(N,u'),...\in\cP$ for the same non-terminal symbol $N$, we will sometimes write
$$N \quad \pro \quad u \quad | \quad u' \quad | \quad ...$$

Production rules allow us to substitute a non-terminal symbol with a sequence of both letters and non-terminal symbols, as follows. Suppose we are given a word $v\in(\cA\cup\cN)^*$ such that the $k$-th symbol of $v$ is $N\in\cN$, and suppose we are given a production rule $(N,u)\in\cP$. Then we take that specific occurrence of $N$ in $v$, and substitute it with $u$, in order to obtain a new word $v'$. In this case, we denote $v\pro_{k,(N,u)}v'$.

\begin{mydef}
Given a context-free grammar $(\cN,\cP,S)$, a \textbf{derivation} is a sequence $$v_0\pro_{k_0,(S,u_0)}v_1\pro_{k_1,(N_1,u_1)}v_2\pro...\pro_{k_{r-1},(N_{r-1},u_{r_1})}v_r$$ with $v_0,v_1,...,v_{r-1},v_r\in(\cA\cup\cN)^*$. In this case we say that $v_0$ \textbf{derives} $v_r$ and we denote $v_0\der v_r$.
\end{mydef}

Given a context-free grammar, we can produce a language; we start with the word given by the single initial symbol $S$, and we inductively apply substitutions, according to the given production rules, until we get a word which doesn't contain any non-terminal symbol. The set of words that can be obtained in this way is the language associated to our context-free grammar.

\begin{mydef}
Let $\cL\subseteq\cA^*$ be a language. Then $\cL$ is called \textbf{context-free language} if there is a context-free grammar $(\cN,\cP,S)$ such that $w\in\cL$ if and only if $S\der w$.
\end{mydef}

Context-free languages are closed under finite union, but not under finite intersection or complement; instead, they are closed under intersection with a regular language.

\begin{myprop}\label{cfreeunion}
If $\cL,\cL'$ are a context-free languages, then $\cL\cup\cL'$ is a context-free language. Moreover there is an algorithm that, given context-free grammars for $\cL,\cL'$, produces a context-free grammar for $\cL\cup\cL'$.
\end{myprop}
\begin{proof}
See Theorem 7.24 of \cite{HMU06}.
\end{proof}

\begin{myprop}\label{cfreeintersection}
If $\cL$ is a context-free language and $\cR$ is a regular language, then $\cL\cap\cR$ is a context-free language. Moreover there is an algorithm that, given a context-free grammar for $\cL$ and a finite automaton for $\cL$, produces a context-free grammar for $\cL\cap\cR$.
\end{myprop}
\begin{proof}
See Theorem 7.27 of \cite{HMU06}.
\end{proof}

Context-free languages are also closed under inverse image. Let $\cA,\cB$ be two alphabets and let $\psi:\cA\rar\cB^*$ be a set map. Then there is a unique extension $\psi:\cA^*\rar\cB^*$ that preserves concatenations (which we still denote by $\psi$ with an abuse of notation).

\begin{myprop}\label{cfreepreimage}
If $\cL\subseteq\cB^*$ is a context-free language, then $\psi^{-1}(\cL)\subseteq\cA^*$ is a context-free language. Moreover there is an algorithm that, given a context-free grammar for $\cL$ and the map $\psi$, produces s context-free grammar for $\psi^{-1}(\cL)$.
\end{myprop}
\begin{proof}
See Theorem 7.30 of \cite{HMU06}.
\end{proof}

The most important property that we want to test about context-free grammar is whether the corresponding context-free language is non-empty.

\begin{myprop}\label{nonempty}
There is an algorithm that, given a context-free grammar, determines whether the corresponding language is non-empty and, if so, produces a word in the language.
\end{myprop}
\begin{proof}
See Section 7.4.3 of \cite{HMU06}.
\end{proof}

\begin{myrmk}
The constructions described in \cite{HMU06} for the proofs of Propositions \ref{cfreeintersection} and \ref{cfreepreimage} are indeed algorithmic, but require the notion of push-down automaton, and involve a process of conversion from context-free grammar to push-down automaton and vice-versa. Later, in Section \ref{SectionAlgorithms}, we will describe constructions, for those grammars specific for our needs, that don't require the notion of push-down automaton; our aim is to optimize the algorithms about equations in order to make them run faster. Section \ref{SectionAlgorithms} also includes a discussion about the algorithm of Proposition \ref{nonempty}.
\end{myrmk}

\subsection{Context-free languages and free groups}\label{Subsectioneta}

Let $F_n$ be a free group with basis $a_1,...,a_n$; we denote with $\ol{w}\in F_n$ the inverse of an element $w\in F_n$. Consider the alphabet $\cA=\{a_1,...,a_n,\ol{a}_1,...,\ol{a}_n\}$ and the set $\cA^*$ of finite words over this alphabet. There is a natural projection map $\eta:\cA^*\rar F_n$ that sends each word in $\cA^*$ to the element of $F_n$ represented by that word.

\begin{mydef}
Let $W\subseteq F_n$ be any subset. Define the language $\cL(W):=\eta^{-1}(W)\subseteq\cA^*$. This means that $\cL(W)$ is the set of all (reduced and unreduced) words that represent an element of $W$.
\end{mydef}

\begin{mydef}\label{cfreesubset}
Let $W\subseteq F_n$ be any subset. We say that $W$ is \textbf{context-free} if $\cL(W)$ is a context-free language.
\end{mydef}

\begin{myrmk}
As it is stated, the above definition seems to depend on the chosen basis for $F_n$. As we will always work with a fixed basis for $F_n$, this would not be an issue for us. However, we still prefer to point out that the above definition is independent on the chosen basis of $F_n$. To be precise, let $g_1,...,g_k$ be any finite set of generators for $F_n$, and let $\cL_{g_1,...,g_k}(W)$ be the set of words in $g_1,...,g_k,\ol{g}_1,...,\ol{g}_k$ that represent an element of $W$: then it easily follows from Proposition \ref{cfreepreimage} that $\cL_{g_1,...,g_k}(W)$ is context-free if and only if $\cL(W)$ is context-free.
\end{myrmk}

\begin{mylemma}\label{trivial}
The trivial subset $\{1\}\subseteq F_n$ is context-free.
\end{mylemma}
\begin{proof}
Consider the grammar given by a single non-terminal symbol $T$ and production rules

\vspace{0.1cm}

\begin{tabular}{l l l p{0.5cm} l}
$T$ & $\pro$ & $\epsilon \quad | \quad TT \quad | \quad a_iT\ol{a}_i \quad | \quad \ol{a}_iTa_i$ && for $i=1,...,n$. \\
\end{tabular}
\end{proof}

It is natural to ask what happens if, instead of considering the language $\cL(W)$ of all words representing an element of $W$, we consider the language of all \textit{reduced} words representing an element of $W$. This corresponds essentially to taking the intersection with a regular language.

\begin{myprop}\label{cfreeequivdef}
Let $W$ be any subset of $F_n$. The the following are equivalent:

(i) The language $\cL(W)$ of all words representing elements of $W$ is context-free.

(ii) The language $\rlan{W}$ of all reduced words representing elements of $W$ is context-free.
\end{myprop}
\begin{proof}
(i)$\Rar$(ii). We take the context-free language $\cL(W)$ and we intersect it with the regular language of all words in $\cA^*$ which are reduced. We obtain the language $\rlan{W}$ which, by Proposition \ref{cfreeintersection}, is context-free.

(ii)$\Rar$(i). Suppose $(\cN,\cP,S)$ is a context-free grammar generating $\rlan{W}$ and let $(\cN_1,\cP_1,T)$ be a grammar for the language $\cL(1)$, for example the one described in Lemma \ref{trivial}. For every production rule $$N\pro u_1...u_l$$ in $\cP$ with $u_1,...,u_l\in(\cA\cup\cN)$, we define a production rule $$N\pro Tu_1Tu_2T...u_{l-1}Tu_lT$$ and we call $\cP'$ the set of all these new production rules. Then the context-free grammar $(\cN\cup\cN_1,\cP'\cup\cP_1,S)$ generates the language $\cL(W)$, as desired.
\end{proof}

In what follows, we will need to deal with cyclically reduced words.

\begin{mydef}
Let $W$ be any subset of $F_n$. Define $\clan{W}\subseteq\cA^*$ to be the language of all cyclically reduced words that represent an element of $W$.
\end{mydef}

If $W$ is a context-free subset of $F_n$, then $\clan{W}$ is a context-free language: this follows from Proposition \ref{cfreeintersection}, intersecting $\rlan{W}$ with the regular language of all words whose first and last letter are not inverse of each other.

\section{Context-free languages and equations over a free group}\label{SectionCfreeIdeals}

Let $F_n$ be a free group on $n$ generators $a_1,...,a_n$. Let $H\sgr F_n$ be a finitely generated subgroup of rank $r$, and let $h_1,...,h_r$ be a basis for $H$. Let $g\in F_n$ be any element.

Let $\fI_g\nor H*\gen{x}$ be the ideal of the equations for $g$ over $H$. We here prove that $\fI_g$ is a context-free subset of $H*\gen{x}$: this provides an algorithm that tells whether $\fI_g$ is trivial or not, and in case also produces a non-trivial equation for $g$ over $H$.

We then turn our attention to the subsets $\fJ_{g,d}\subseteq\fI_g$ of the equations of a certain fixed degree $d\ge1$.
We prove that these subsets of $H*\gen{x}$ are context-free too: this gives an algorithm that tells us whether there is a non-trivial equation of degree $d$.

Finally, we provide an algorithm to find the minimum possible degree for an equation in $\fI_g$.

\subsection{Ideals of equations are context-free}

\begin{mythm}\label{cfreeIg}
The set $\fI_g$ is context-free as a subset of $H*\gen{x}$.
\end{mythm}
\begin{proof}
Consider the alphabets $\cA=\{a_1,\ol{a}_1,...,a_n,\ol{a}_n\}$ and $\cB=\{h_1,\ol{h}_1,...,h_r,\ol{h}_r,x,\ol{x}\}$ where $h_1,...,h_r$ is a basis for $H$. Consider a map $\psi:\cB^*\rar\cA^*$ such that $\psi$ sends each element of $\cB$ to a word representing that element in the free group $F_n$ (and $x$ and $x^{-1}$ to words representing $g$ and $g^{-1}$ respectively), and such that $\psi$ preserves concatenations. By Lemma \ref{trivial} we have that $\cL(1)\subseteq\cA^*$ is context-free. By Proposition \ref{cfreepreimage} we obtain that $\psi^{-1}(\cL(1))\subseteq\cB^*$ is context-free. But $\psi^{-1}(\cL(1))=\cL(\fI_g)$ and the conclusion follows.
\end{proof}

\begin{myrmk}
Let $N$ be a normal subgroup of a finitely generated free group $F$. We point out that $N$ being context-free is the same as the quotient $F/N$ being virtually free. This is a result of D. E. Muller and P. E. Schupp; for more details see \cite{MS83}.
\end{myrmk}

\begin{mythm}\label{cfreeIg2}
There is an algorithm that, given $H\sgr F_n$ finitely generated and $g\in F_n$, tells us whether $\fI_g$ contains a non-trivial equation or not and, in case it does, produces a non-trivial equation in $\fI_g$.
\end{mythm}
\begin{proof}
We proceed as in the proof of Proposition \ref{cfreeIg}. Lemma \ref{trivial} provides an explicit grammar for the language $\cL(1)\subseteq\cA^*$. According to Proposition \ref{cfreepreimage}, we can algorithmically build a grammar for the language $\cL(\fI_g)$. It is easy to construct a finite automaton for the regular language $\cR$ of all non-trivial reduced words in $\cB^*$: by Proposition \ref{cfreeintersection} we can algorithmically obtain a context-free grammar for the language $\cL(\fI_g)\cap\cR$ of all non-trivial reduced words representing an equation in $\fI_g$. We now apply to this grammar the algorithm of Proposition \ref{nonempty}, and the conclusion follows.
\end{proof}

\subsection{Equations of a certain fixed degree}

Recall that the degree of an equation $w(x)\in H*\gen{x}$ is defined as the number of occurrences of $x$ and $\ol{x}$ in the cyclic reduction of $w(x)$. The cyclic reduction has to be taken with respect to the basis $h_1,...,h_r,x$ for $H*\gen{x}$, where $h_1,...,h_r$ is a basis for $H$. Notice that choosing a different basis for $H$ can change the cyclic reduction of an equation $w(x)$, but it will not change the number of occurrences of $x$ and $\ol{x}$, so that the degree of $w(x)$ is well defined. For $d\in\bN$, we denote by $\fJ_{g,d}$ the set of all equations in $\fI_g$ that have degree exactly $d$.

\begin{mythm}\label{cfreeJgd}
The set $\fJ_{g,d}$ is context-free as a subset of $H*\gen{x}$.
\end{mythm}
\begin{proof}
Let $h_1,...,h_r$ be a basis for $H$, and we work with the alphabet $\cB=\{h_1,\ol{h}_1,...,h_r,\ol{h}_r,x,\ol{x}\}$. By Theorem \ref{cfreeIg} we have that the language $\cL(\fI_g)$ is context-free. It is easy to construct a finite automaton for the language $\cR$ of all cyclically reduced words that contain exactly $d$ occurrences of $x$ and $\ol{x}$. By Proposition \ref{cfreeintersection} we obtain that the language $\clan{\fJ_{g,d}}=\cL(\fI_g)\cap\cR$ is context-free. Let $(\cN,\cP,S)$ be a context-free grammar for $\clan{\fJ_{g,d}}$ and let $\cP'$ be the set of production rules given by

\vspace{0.1cm}

\begin{tabular}{l l l p{0.5cm} l}
$S$ & $\pro$ & $xS\ol{x} \quad | \quad \ol{x}Sx \quad | \quad h_iS\ol{h}_i \quad | \quad \ol{h}_iSh_i$ && for $i=1,...,r$. \\
\end{tabular}

\vspace{0.1cm}

Let $\cL'$ be the language generated by the grammar $(\cN,\cP\cup\cP',S)$: we claim that $\rlan{\fJ_{g,d}}\subseteq\cL'\subseteq\cL(\fJ_{g,d})$. The second inclusion is trivial. For the first inclusion, suppose $w$ is a reduced word representing an equation of degree $d$: then we can write $w=uv\ol{u}$ where $v$ is a cyclically reduced word representing an equation of degree $d$, $u$ is a reduced word and $\ol{u}$ is the formal inverse of $u$. But then we can obtain a derivation $S\der uS\ol{u}$ using only production rules in $\cP'$, and a derivation $uS\ol{u}\der uv\ol{u}$ using only production rules in $\cP$. It follows that $w=uv\ol{u}$ belongs to $\cL'$.

Let now $\cR'$ be the regular language of all reduced words, and by Proposition \ref{cfreeintersection} we have that $\rlan{\fJ_{g,d}}=\cL'\cap\cR'$ is context-free. The conclusion follows from Proposition \ref{cfreeequivdef}.
\end{proof}

\begin{mythm}\label{cfreeJgd2}
There is an algorithm that, given $H\sgr F_n$ finitely generated and $g\in F_n$ and $d\in\bN$, tells us whether the set $\fJ_{g,d}$ is non-empty and, if so, produces an element in $\fJ_{g,d}$.
\end{mythm}
\begin{proof}
We proceed as in the proof of Theorem \ref{cfreeIg2}. In the proof of Theorem \ref{cfreeIg2} we proved that it is possible to algorithmically build a grammar for the language $\cL(\fI_g)$. It is quite easy to produce an algorithm that constructs a finite automaton for the regular language $\cR$ of all cyclically reduced words that contain exactly $d$ occurrences of $x$ and $\ol{x}$. By Proposition \ref{cfreeintersection} we obtain an algorithm to construct a grammar for the language $\cL(\fI_g)\cap\cR$. We apply the algorithm of Proposition \ref{nonempty} to this grammar and the conclusion follows.
\end{proof}

\section{Asymptotic growth rate of the number of equations}\label{SectionAsymptotic}

The aim of this section is to study the growth rate of the number of equations having $g$ as a solution up to a given length. We deal with this question both for the set $\fI_g$ of all equations, and for the set $\fJ_{g,d}$ of equations of a certain given degree $d$.

\begin{mydef}
Let $\rho:[0,+\infty)\rar[0,+\infty)$ be a non-decreasing function. We say that $\rho$ has

(i) \textbf{exponential growth} if $e^{\alpha M}\le \rho(M)\le e^{\beta M}$ for some $\alpha,\beta>0$ and for all $M$ big enough.

(ii) \textbf{polynomial growth of degree $k$}, where $k\ge0$ is an integer, if $\alpha M^k\le \rho(M)\le \beta M^k$ for some $\alpha,\beta>0$ and for all $M$ big enough.
\end{mydef}

Polynomial growth of degree $1$ is usually called \textbf{linear growth}.

\subsection{Growth rate of context-free languages}

Let $\cA$ be an alphabet and let $\cL\subseteq\cA^*$ be a language. The growth function is defined by
$$\rho_\cL(M)=\abs{\{w\in\cL : l(w)\le M\}}$$
We are interested in the growth of the function $\rho_\cL$ when $\cL$ is a context-free language. This has already been studied, for example in \cite{BG02} and in \cite{Inc01}, and the possibilities are described in full in \cite{GKRS08}; we restate here their result.

\begin{mythm}[See \cite{GKRS08}]\label{cfreegrowth}
Let $\cL$ be a context-free language. Then exactly one of the following takes place:

(i) The function $\rho_\cL$ has polynomial growth of degree $k$ for some $k\in\bN$.

(ii) The function $\rho_\cL$ has exponential growth.

Moreover, there is an algorithm that, given a context-free grammar for $\cL$, tells us (in polynomial time) which of the above cases takes place, and, in case (i), also tells us the degree $k$ of the growth.
\end{mythm}

Let now $F_n$ be a free group on $n$ generators $a_1,...,a_n$ and consider as usual the alphabet $\cA=\{a_1,...,a_n,\ol{a}_1,...,\ol{a}_n\}$. If $\cL(1)$ is the language of all words representing the identity element, then the function $\rho_{\cL(1)}$ has exponential growth. In particular, for every non-empty subset $S\subseteq F_n$ we have that $\rho_{\cL(S)}$ has exponential growth. Essentially, here the exponential growth is given by the fact that each element of $F_n$ can be represented by a lot of words; in order to eliminate ambiguity, in what follows we will consider only cyclically reduced words.

\subsection{Growth rate of sets of equations}\label{SubsectionGrowth}

Let $H\sgr F_n$ be a finitely generated subgroup with basis $h_1,...,h_r$ and let $g\in F_n$ be an element such that $\fI_g$ is non-trivial. We work over the alphabet $\cB=\{h_1,\ol{h}_1,...,h_n,\ol{h}_n,x,\ol{x}\}$. We want to study the growth of the function $\rho_g(M):=\rho_{\clan{\fI_g}}(M)$ which counts the number of cyclically reduced equations having $g$ as a solution up to a certain length, and of the functions $\rho_{g,d}(M):=\rho_{\clan{\fJ_{g,d}}}(M)$ which count the number of cyclically reduced equations of degree $d$ having $g$ as a solution up to a certain length.

\begin{mythm}\label{growth1}
The function $\rho_g(M)$ has exponential growth.
\end{mythm}
\begin{proof}
The language $\clan{\fI_g}$ is context-free, so by Theorem \ref{cfreegrowth} it has either polynomial growth (of some degree $k\in\bN$) or exponential growth. To prove the theorem, it is enough to find an exponential lower bound.

Assume that $H$ has rank at least $2$. Let $w\in\clan{\fI_g}$ and let $t_1,t_2,t_3,t_4\in\cB$ be the first and the last letter of $w$ and their inverses. Choose a reduced word $h\in\cB^*$ which doesn't begin or end with any of $t_1,t_2,t_3,t_4$: then we have a word $hw\ol{h}w\in\clan{\fI_g}$. Different choices of $h$ give different elements $hw\ol{h}w$. Since we have exponentially many choices for $h$ of length at most $M$, we obtain an exponential lower bound on $\rho_g$.

Assume now that $H$ has rank $1$. This means that $H=\gen{h}$ where $h=b^m$ and $g=b^k$ for some $b\in F_n\setminus\{1\}$ and $m\in\bZ\setminus\{0\}$ and $k\in\bZ$. In this case $\cB=\{h,\ol{h},x,\ol{x}\}$. Fix $M$ and for each $(\alpha_1,...,\alpha_{2M-1})\in\{1,2\}^{2M-1}$ define the word $w_{\alpha_1,...,\alpha_{2M-1}}\in\cB^*$ given by
$$w_{\alpha_1,...,\alpha_{2M-1}}=h^{\delta_1\alpha_1}xh^{\delta_2\alpha_2}\ol{x}h^{\delta_3\alpha_3}xh^{\delta_4\alpha_4}\ol{x}\ ...\ h^{\delta_{2M-1}\alpha_{2M-1}}xh^{-(\delta_1\alpha_1+...+\delta_{2M-1}\alpha_{2M-1})}\ol{x}$$
where $\delta_1=1$ and $\delta_j\in\{+1,-1\}$ is defined inductively by setting $\delta_j=1$ if and only if we have $\delta_1\alpha_1+...+\delta_{j-1}\alpha_{j-1}\le0$. We point out that $h^e$ for $e$ negative integer has to be interpreted as the word $\ol{h}^{-e}$. It is easy to check that $w_{\alpha_1,...,\alpha_{2M-1}}$ belongs to $\clan{\fI_g}$, and that different choices of $(\alpha_1,...,\alpha_{2M-1})$ produce different words. It is easy to show by induction that $\abs{\delta_1\alpha_1+...+\delta_j\alpha_j}\le2$ for $j=1,...,2M-1$, and thus it follows that $w_{\alpha_1,...,\alpha_{2M-1}}$ has length at most $6M$. Since we have $2^{2M-1}$ choices for $(\alpha_1,...,\alpha_{2M-1})$, this provides an exponential lower bound for $\rho_g(M)$.
\end{proof}

\begin{mythm}\label{growth2}
Let $d\in\bN$ be a non-negative integer. Then the function $\rho_{g,d}(M)$ has either exponential growth, or polynomial growth of degree $k$ for some $k\in\bN$. Moreover, there is an algorithm that tells us which case takes place, and in the second case computes the degree $k$ of the growth.
\end{mythm}
\begin{proof}
As in the proof of Theorem \ref{cfreeJgd2}, we can produce a grammar for the context-free language $\clan{\fJ_{g,d}}$. We apply Theorem \ref{cfreegrowth} to this grammar. The conclusion follows.
\end{proof}

\subsection{Degrees with polynomial growth rate}

Let $H\sgr F_n$ be a finitely generated subgroup with basis $h_1,...,h_r$ and let $g\in F_n$ be an element such that $\fI_g$ is non-trivial. We work over the alphabet $\cB=\{h_1,\ol{h}_1,...,h_n,\ol{h}_n,x,\ol{x}\}$. Recall that the degree of an equation $w(x)\in H*\gen{x}$ is defined as the number of occurrences of $x$ and $\ol{x}$ in the cyclic reduction of $w(x)$.

\begin{mydef}\label{defDg}
Define the following sets:

(i) $D_g=\{d\in\bN :\text{ there is a non-trivial equation }w\in\fI_g\text{ of degree }d\}$.

(ii) $D_g^{\pol,k}=\{d\in D_g : \rho_{g,d}$ has polynomial growth of degree $k\}$.

(iii) $D_g^\ex=\{d\in D_g : \rho_{g,d}$ has exponential growth$\}$.
\end{mydef}

According to Theorem \ref{growth2}, we have a partition $D_g=D^\ex_g\sqcup\bigsqcup_{k\in\bN}D^{\pol,k}_g$.

\begin{mylemma}\label{expdegrees}
Suppose $H$ has rank at least $2$. Then for each $d,d'\in D_g$ and $k\ge0$ we have $d+d'+2k\in D^\ex_g$.
\end{mylemma}
\begin{myrmk}
Compare this with with the results of Section 5 of \cite{PART1}.
\end{myrmk}
\begin{proof}
Let $w\in\clan{\fJ_{g,d}}$ and, up to cyclic permutation, we can assume that $w$ is of the form $c_1x^{e_1}...c_\alpha x^{e_\alpha}$ with $c_1,...,c_\alpha\in\{h_1,\ol{h}_1,...,h_k,\ol{h}_k\}^*$ and $e_1,...,e_\alpha\in\bZ\setminus\{0\}$. Similarly let $w'\in\clan{\fJ_{g,d'}}$ where $w'=c_1'x^{e_1'}...c_{\beta}'x^{e_\beta'}$ with $c_1',...,c_\beta'\in\{h_1,\ol{h}_1,...,h_k,\ol{h}_k\}^*$ and $e_1',...,e_\beta'\in\bZ\setminus\{0\}$. Without loss of generality we assume $e_\beta'>0$.

We fix $k\ge0$ and for each $h\in\{h_1,\ol{h}_1,...,h_r,\ol{h}_r\}^*$ we consider the word $w_h=\ol{h}wh\ol{x}^kw'x^k\in\cB^*$, which represents an equation for $g$. If the $h$ is reduced, and the first letter of $h$ is different from the first of $c_1$, and the last letter of $h$ is not the inverse of the first of $c_1'$, then the word $w_h$ is cyclically reduced, and in particular it belongs to $\clan{\fJ_{g,d}}$. Since $H$ has rank $r\ge2$, there are exponentially many choices for such an $h$ of a given length and with those properties. This proves the desired result.
\end{proof}

As a corollary of Lemma \ref{expdegrees}, we have that $D_g$ coincides, up to a finite set, with either the set of natural numbers or the set of non-negative even numbers (see also $\Pdegreeset$ of \cite{PART1}). If $D_g$ contains an odd degree $d$, then it contains all the degrees $2d+2k$ and $3d+2k$ for $k\in\bN$, and thus coincides with the whole $\bN$ up to a finite set. If $D_g$ contains only even degrees, and if it contains a degree $d$, then it contains all the degrees $2d+2k$ for $k\in\bN$, and thus coincides with $2\bN$ minus a finite set. The following Lemma \ref{allexp} tells us that the same holds for $D^\ex_g$.

\begin{mylemma}\label{allexp}
Suppose $H$ has rank at least $2$. Then $D_g\setminus D^\ex_g$ is finite.
\end{mylemma}
\begin{proof}
If $D_g$ only contains even numbers, then take $d\in D_g$. By Lemma \ref{expdegrees}, every even number $\ge 2d$ belongs to $D^\ex_g$, and we are done. Suppose now $D_g$ contains an odd number $d\in D_g$: by Lemma \ref{expdegrees} we have that $D^\ex_g$ contains all even numbers $\ge 2d$ and all odd numbers $\ge 3d$, and we are done.
\end{proof}

\begin{mythm}\label{Dgpolk}
Suppose $H$ has rank at least $2$. Then there is an algorithm that computes the finite sets $D^{\pol,k}_g$ for $k\in\bN$; all but finitely many of these are empty.
\end{mythm}
\begin{proof}
As proved in \cite{PART1}, we can produce algorithmically a finite set of normal generators $\fI_g=\ggen{w_1,...,w_k}$ for $\fI_g\nor H*\gen{x}$.

Suppose $w_1,...,w_k$ all have even degree: then $D_g$ only contains even numbers, according to \cite{PART1}. We take $d_1$ to be the degree of $w_1$, and we apply Theorem \ref{growth2} to check the type of growth of $\rho_{g,d}$ for all $d<2d_1$; as in Lemma \ref{allexp} we have that every even degree $\ge 2d_1$ belongs to $D^\ex_g$, and so we are done.

Suppose one of $w_1,...,w_k$ has odd degree: let's say $w_1$ has degree $d_1$ odd. We apply Theorem \ref{growth2} to check the type of growth of $\rho_{g,d}$ for all $d<3d_1$; as in Lemma \ref{allexp} we have that each degree $\ge 3d_1$ belongs to $D^\ex_g$, and so we are done.
\end{proof}

\section{Running time of the algorithms}\label{SectionAlgorithms}

In this section we are interested in bounding the running time of the algorithms described in this paper. As we already pointed out in the introduction, we assume that we are working on a RAM machine; we assume that the machine can perform in time $O(1)$ the most basic operations, including sum and multiplication of two integer numbers. Some of the algorithms have context-free grammars as input or output, and thus we need a way to quantify the size of a context-free grammar; we express the complexity of a context-free grammar in terms of the following parameters:

\begin{mydef}
Let $(\cN,\cP,S)$ be a context-free grammar.
%
%

(i) We denote with $\norma{\cP}=\sum_{(N,u)\in\cP}(1+l(u))$ the \textbf{total size} of the grammar.

(ii) We denote with $\ram{\cP}$ the maximum number of non-terminal symbol (counted with repetitions) in $u$ for $(N,u)\in\cP$.
\end{mydef}

\begin{myrmk}
Here the notation $\ram{\cP}$ is for ``ramification'' of $\cP$.
\end{myrmk}

\subsection{Checking whether a context-free language is non-empty}

\newcommand{\true}{\texttt{true}}
\newcommand{\false}{\texttt{false}}
\newcommand{\occorrenze}{\texttt{occurrences}}
\newcommand{\termina}{\texttt{terminates}}
\newcommand{\lunghezza}{\texttt{length}}
\newcommand{\parola}{\texttt{word}}
\newcommand{\rimasti}{\texttt{remaining}}
\newcommand{\stazza}{\texttt{size}}
\newcommand{\coda}{\texttt{queue}}
\newcommand{\lista}{\texttt{list}}
\newcommand{\update}{\texttt{Update}}

We now provide a classical algorithm to check whether a context-free grammar produces the empty language or not, taken from \cite{HMU06}; more generally, we provide algorithms to produce words in the language generated by a given context-free grammar.

The following Proposition \ref{nonempty2} below is a restatement of Proposition \ref{nonempty}, but with an additional bound on the running time of the algorithm. The algorithm described in the proof of Proposition \ref{nonempty2} below is already described in Section 7.4.3 of \cite{HMU06}; however, we need to make the proof more explicit in order to keep track of the complexity of the variants we need.

\begin{myprop}\label{nonempty2}
There is an algorithm that, given a context-free grammar $(\cN,\cP,S)$, tells us whether the corresponding context-free language is non-empty. The algorithm runs in time $O(\norma{\cP})$.
\end{myprop}
\begin{proof}[Description of the algorithm]
We assume that the non-terminal symbols are numbered from $0$ to $\abs{\cN}-1$, with the initial symbol $S$ being indexed as $0$, and that the production rules are numbered from $0$ to $\abs{\cP}-1$. We initialize the following data structures:

(i) An array $\occorrenze$ of length $\abs{\cN}$. For $i=0,...,\abs{\cN}-1$, the entry $\occorrenze[i]$ is the list of the occurrences of the $i$-th non-terminal symbol in the production rules: the list $\occorrenze[i]$ contains a couple of integers $(c,d)$ if and only if the $c$-th production rule is $N\pro u$ and the $d$-th symbol of $u$ is the $i$-th non-terminal symbol. We read the list of all production rules once, and whenever we find a non-terminal symbol, we add the corresponding couple to the suitable entry of $\occorrenze$.

(ii) An array $\termina$ of length $\abs{\cN}$, each entry containing either $\true$ or $\false$. We initialize $\termina[i]=\false$ for $i=0,...,\abs{\cN}-1$.

(iii) An array $\rimasti$ of length $\abs{\cP}$. For $j=0,...,\abs{\cP}-1$, we consider the $j$-th production rule $N\pro u$, and we initialize $\rimasti[j]$ to be the number of non-terminal symbols in $u$, counted with repetition.

(iv) A queue $\coda$ that contains the indices of some production rules. We use a first-in first-out queue (but it doesn't really matter). We initialize $\coda$ by reading the array $\rimasti$ once, and whenever we find $\rimasti[j]=0$ we add $j$ to $\coda$.

We now describe the $\update$ step. We take (and remove) an index $j$ from $\coda$, and we consider the $j$-th production rule $(N,u)$, where $N$ is the $i$-th non-terminal symbol. If $\termina[i]$ is $\true$ we immediately terminate the $\update$ step; otherwise we set $\termina[i]$ from $\false$ to $\true$ and we proceed. We read the list $\occorrenze[i]$: when we read the couple $(c,d)$ we decrease by $1$ the number $\rimasti[c]$, and if $\rimasti[c]$ becomes $0$, then we add $c$ to $\coda$.

The algorithm now runs as follows: it initializes the data structures (i),(ii),(iii),(iv) as explained above, and then it starts running the $\update$ step until $\coda$ becomes empty. At that point the algorithm stops and gives $\termina[0]$ as output.
\end{proof}
\begin{proof}[Proof that the algorithm works and bounds on the running time]\

(i) The array $\occorrenze$ is initialized at the beginning of the algorithm and then never modified.

(ii) At any point during the algorithm, $\termina[i]$ is $\true$ if and only if we have found out that the $i$-th non-terminal symbol can produce a word in $\cA^*$. Symbols that have been inserted in $\coda$ but have not yet gone through the $\update$ step still remain on $\false$.

(iii) At any point during the algorithm, if $N\pro u$ is the $j$-th production rule, then $\rimasti[j]$ is the number of non-terminal symbols in $u$ such that we don't know whether they can produce a word in $\cA^*$ or not. In fact, at the beginning of the algorithm we have that the $i$-th symbol satisfies $\termina[i]=\false$ for $i=0,...,\abs{\cN}-1$, and thus we initialize $\rimasti$ by counting all the non-terminal symbols. At each $\update$ step, we update $\termina[i]=\true$ exactly at the same time when we update $\rimasti$; we decrease the entries of $\rimasti$ according to the occurrences of the $i$-th non-terminal symbol, which we read from $\occorrenze$.

(iv) The queue $\coda$ contains the indices of the production rules that have been discovered to produce a word in $\cA^*$; such a production $N\pro u$ tells us that the symbol $N$ can produce a word in $\cA^*$, and thus we have to update the other data structures. During the initialization, we add to $\coda$ exactly the non-terminal symbols $N$ which have some production rule $N\pro u$ with $u\in\cA^*$, since this proves that each of them can produce a word in $\cA^*$. During the update process, whenever the $j$-th production rule $N\pro u$ satisfies $\rimasti[j]=0$, this means that all the non-terminal symbols in $u$ have been proved to produce some word in $\cA^*$, and thus we have to add $j$ to $\coda$. Notice that each production rule is added to $\coda$ at most once, when the corresponding entry of $\rimasti$ becomes $0$.

The $\update$ step consists of taking a production rule $N\pro u$ from $\coda$ and updating the other data structures adding the information that $N$ can produce a word in $\cA^*$. If $N$ had already been found out to produce a word in $\cA^*$, then we just skip to updating the next production rule in $\coda$; otherwise we modify the value of $\termina$ accordingly, and then we read from $\occorrenze$ the occurrences of $N$ in the production rules, and we update the array $\rimasti$, taking into account that $N$ no longer counts as non-terminal, as it can be substituted with a word in $\cA^*$. It is possible that some entries of $\rimasti$ decrease to $0$ in the update process, meaning that other production rules have to be added to $\coda$.

This shows that, at the end of the algorithm, if $\termina[i]=\true$ then the $i$-th non-terminal symbol can produce a word in $\cA^*$. We want to prove that the converse implication holds. Suppose by contradiction that the algorithm ended, and the $i$-th non-terminal symbol $N$ can produce a word in $\cA^*$ but still has $\termina[i]=\false$. We take a derivation $N\pro v_1\pro...\pro v_r$ with $v_r\in\cA^*$ and without loss of generality we can assume that both $N$ and the derivation have been chosen with $r$ minimum possible. The minimality of $r$ implies that every non-terminal symbol in $v_1$ has the corresponding entry of $\termina$ equal to $\true$. The discussion above implies that, at the end of the algorithm, the production rule $(N,v_1)$ has the corresponding entry of $\rimasti$ equal to $0$. But this means that the production rule had been added to $\coda$, and thus that $\termina[i]$ had been updated to $\true$, contradiction.

Thus at the end of the algorithm we have that the array $\termina$ tells us, for each symbol, whether they can produce a word in $\cA^*$ or not. In particular, as the initial symbol $S$ is indexed as $0$, we get the correct output $\termina[0]$.

Let's now discuss the complexity of the algorithm. The initialization of $\occorrenze$, $\termina$, $\rimasti$, $\coda$ are done in time $O(\norma{\cP})$, $O(\abs{\cN})$, $O(\norma{\cP})$, $O(\abs{\cP})$ respectively. During all the update cycles, the operations of taking one production rule from $\coda$, checking and updating the value of $\termina$, and adding a new production rule to $\coda$, are done at most once for each production rule, so the total time required is $O(\abs{\cP})$; the operation of updating $\rimasti$ due to a given occurrence $(c,d)$ of a certain non-terminal symbol is done at most once for each symbol in each of the words $u$ for $(N\pro u)\in\cP$, so the total time required is $O(\norma{\cP})$.
\end{proof}

When the algorithm of Proposition \ref{nonempty2} gives an affirmative answer, we would like to algorithmically and efficiently build a word belonging to the language; however, we point out that the length of the shortest word in the language can be exponential in $\norma{\cP}$, and thus any algorithm that explicitly writes down such a word runs in exponential time. One possible way of going around this problem is to give as output, instead of the full explicit word belonging to the language, some other piece of information, that allows us to build such word in a straightforward manner, while at the same time being shorter.

\begin{myprop}\label{explicit1}
There is an algorithm that, given a context-free grammar $(\cN,\cP,S)$ with non-empty language $\cL$, produces a list of production rules $(N_1\pro u_1),...,(N_m\pro u_m)\in\cP$ such that:

(i) The symbols $N_1,...,N_m$ are pairwise distinct and $N_m=S$ is the initial symbol.

(ii) The word $u_i$ contains only terminal symbols and possibly the symbols $N_1,...,N_{i-1}$.

In particular, starting at $S$ and applying the production rules in the list (in any order) produces a word in the language $\cL$. The algorithm runs in time $O(\norma{\cP})$.
\end{myprop}

\begin{myrmk}
In particular, the algorithm of the above Proposition \ref{explicit1} produces a sub-grammar
$$(\{N_1,...,N_m\},\{N_1\pro u_1,...,N_m\pro u_m\},S)$$
of our given grammar $(\cN,\cP,S)$. This sub-grammar has the special property that its language consists of a single word; such a grammar is sometimes referred to as \textit{straight line grammar}.
\end{myrmk}

\begin{proof}[Proof of Proposition \ref{explicit1}]
We run the same algorithm as in the proof of Proposition \ref{nonempty2} but with a few changes.

(i) We initialize an array $\lista$, which at the beginning is empty, but that will contain the list of (the indices of) the production rules that is required as output. During the $\update$ step, if we have taken the index $j$ from $\coda$ and if we update $\termina[i]$ from $\false$ to $\true$, then at the same time we also add $j$ to the end of $\lista$.

(ii) At the end of the algorithm we give as output the array $\lista$, truncated at the (unique) occurrence of a production rule for the symbol $S$.
\end{proof}

We are now interested in producing a word in the language that has minimum \textbf{size}. In order to keep flexible our notion of \textbf{size}, we assume we are given a function $\sigma:\cA\rar\bN$ that assigns to each letter $a$ in our alphabet a non-negative integer number $\sigma(a)$; for a word $w=w_1...w_{l(w)}\in\cA^*$ we define $\sigma(w):=\sum_{i=1}^{l(w)}\sigma(w_i)$.

\begin{mydef}\label{defsize}
Let $(\cN,\cP,S)$ be a context-free grammar.

(i) For a non-terminal symbol $N\in\cN$ define $\sigma(N):=\min\{\sigma(w) : w\in\cA^*\text{ with }N\der w\}$.

(ii) For a word $u=u_1...u_{l(u)}\in(\cA\cup\cN)^*$ define $\sigma(u):=\sum_{i=1}^{l(u)}\sigma(u_i)$.
\end{mydef}

\begin{myprop}\label{explicit2}
There is an algorithm that, given a context-free grammar $(\cN,\cP,S)$ with non-empty language $\cL$, and a function $\sigma:\cA\rar\bN$, produces a list of couples $(N_1\pro u_1,\tau_1),...,(N_m\pro u_m,\tau_m)\in\cP\times\bN$ such that:

(i) The symbols $N_1,...,N_m$ are pairwise distinct and $N_m=S$ is the initial symbol.

(ii) We have $\sigma(N_i)=\tau_i$.

(iii) The word $u_i$ contains only terminal symbols and possibly the symbols $N_1,...,N_{i-1}$.

(iv) The sum of $\sigma(c)$, for each symbol $c$ in $u_i$, is equal to $\tau_i$.

In particular $\tau_m=\sigma(S)$ is the minimum possible size for a word in the language. Moreover, starting at $S$ and applying the production rules in the list (in any order) produces a word of minimum possible size in the language $\cL$. The algorithm runs in time $O(\norma{\cP}\log\abs{\cP})$.
\end{myprop}
\begin{proof}
We run the same algorithm as in the proof of Proposition \ref{nonempty2} but with a few changes.

(i) The queue $\coda$ contains couples $(j,\tau)$ where $j$ is the index of a production rule and $\tau\in\bN$. Instead of a first-in first-out queue, this time the elements of $\coda$ are ordered in increasing order on $\tau$. Whenever we insert a new couple to $\coda$, we make sure to preserve the ordering; whenever we extract an element from $\coda$, we take one with $\tau$ smallest.

(ii) We initialize an array $\lista$, which at the beginning is empty, but that will contain the list that is required as output. During the $\update$ step, if we have taken the couple $(j,\tau)$ from $\coda$ and we update $\termina[i]$ from $\false$ to $\true$, then at the same time we also add the couple $(j,\tau)$ to $\lista$.

(iii) We initialize an array $\stazza$ of length $\abs{\cP}$: if the $j$-th production rule is $N\pro u$ then we initialize $\stazza[j]$ to be the sum of $\omega(c)$ for each terminal symbol $c$ in $u$ (with repetition). During the $\update$ step, if we have taken the couple $(j,\tau)$ from $\coda$, whenever we decrease $\rimasti[c]$ by one we also increase $\stazza[c]$ by $\tau$. If $\rimasti[c]$ goes down to zero, then we add the couple $(c,\stazza[c])$ to $\coda$.

(iv) At the end of the algorithm we give as output the array $\lista$, possibly truncated at the (unique) occurrence of a production rule for the symbol $S$.

Each production rule is added at most once to $\coda$, and the steps of adding a couple to $\coda$ is done in time $O(\log\abs{\cP})$. The rest of the operations of the algorithms are performed in total time $O(\norma{\cP})$. Thus the total running time of the algorithm is $O(\norma{\cP}\log\abs{\cP})$.
\end{proof}

If, for any reason, we have information that the language contains at least one word of bounded length, then there is an algorithm to explicitly write down such a word.

\begin{myprop}\label{explicit3}
There is an algorithm that, given a context-free grammar $(\cN,\cP,S)$ and an integer $r\ge0$, produces the following:

(i) A list of all non-terminal symbols $N\in\cN$ that can derive a word of length at most $r$.

(ii) For each $N$ in the list (i), the minimum possible length $l_N$ and a word $w_N\in\cA^*$ of length $l_N$ such that $w_N$ can be produced by means of a derivation starting at $N$.

The algorithm produces the list (i) and the lengths $l_N$ in time $O(\norma{\cP}r)$ and the words $w_N$ in time $O(\norma{\cP}r^2)$.
\end{myprop}
\begin{proof}
The algorithm is again similar to the one in the proof of Proposition \ref{nonempty2}, but with a few differences.

(i) We initialize $\occorrenze$ exactly as in Proposition \ref{nonempty2}.

(ii) We create three arrays $\termina$ and $\lunghezza$ and $\parola$ of length $\abs{\cN}$. For $i=0,...,\abs{\cN}-1$ we initialize $\termina[i]=\false$ and $\lunghezza[i]=0$ and $\parola[i]$ to be the empty string.

(iii) We create two arrays $\rimasti$ and $\stazza$ of length $\abs{\cP}$. For $j=0,...,\abs{\cP}-1$, if $(N,u)$ is the $j-th$ production rule, then we initialize $\rimasti[j]$ to be the number of non-terminal symbols in $u$ and $\stazza[j]$ to be the number of terminal symbols in $u$.

(iv) We create a queue $\coda$ that will contain some indices of some production rules. We initialize $\coda$ to be empty.

(v) We create an integer $s$ and we initialize $s=0$.

We now describe the $\update$ step. We take (and remove) an index $j$ from $\coda$, and we consider the $j$-th production rule $(N,u)$, where $N$ is the $i$-th non-terminal symbol. If $\termina[i]$ is $\true$ we immediately terminate the $\update$ step, otherwise we proceed. We set $\termina[i]=\true$, $\lunghezza[i]=s$ and $\parola[i]$ to be the word obtained from $u$ by substituting each occurrence of each non-terminal symbol with the corresponding entry of $\parola$ (i.e. we read $u$ and whenever we encounter an occurrence of the $i'$-th non-terminal symbol we substitute that occurrence with $\parola[i']$). We read the list $\occorrenze[i]$: when we read the couple $(c,d)$, we decrease by $1$ the number $\rimasti[c]$, we increase by $\lunghezza[i]$ the number $\stazza[c]$, and if $\rimasti[c]$ is $0$ and $\stazza[j]$ is $s$ then we add $c$ to $\coda$.

The algorithm now runs as follows: it initializes the data structures (i), (ii), (iii), (iv), (v) as described above. Then for each $s=0,...,r$ it runs the following: it initializes $\coda$ by adding all the indices $j$ such that $\rimasti[j]=0$ and $\stazza[j]=s$, and then it runs the $\update$ step until $\coda$ becomes empty. At the end of this process, the algorithm terminates and outputs the list of all non-terminal symbols such that $\termina[i]=\true$, and for each of them the integer $\lunghezza[i]$ and the string $\parola[i]$.

At any point during the algorithm, if $\termina[i]$ is $\true$ then $\lunghezza[i]$ and $\parola[i]$ contain the minimum length of a word and a word of that length that can be produced from the $i$-th non-terminal symbol. At any point during the algorithm, if the $j$-th production rule is $(N,u)$, then $\rimasti[j]$ is the number of non-terminal symbols that are contained in $u$ and still have $\termina$ equal to $\false$; similarly, $\stazza[j]$ is the number of terminal symbols in the word $u$ to which we sum also, for each occurrence of a non-terminal symbol with $\termina$ equal to $\true$, the length $\lunghezza$ of the shortest word that we can produce from that symbol. The idea is that during the iteration with a certain value of $s$, we find all the non-terminal symbols that can produce a word of length at most $s$; whenever we increase $s$ to $s+1$ we reset $\coda$ by rechecking all the production rules with $\rimasti$ equal to $0$, because it is possible that one of these production rules has $\rimasti$ equal to $0$ but $\stazza$ equal to $s+1$, so that it has been ignored up to this point. The proof that the algorithm works and the computation of the complexity are completely analogous to the proof of Proposition \ref{nonempty2}.
\end{proof}

\subsection{Unambiguous grammars}

We are interested in context-free grammars with the additional property of being unambiguous. This means that each word in the corresponding language can be obtained in an essentially unique way; this is made precise as follows.

\begin{mydef}
Let $(\cN,\cP,S)$ be a context-free grammar. We say that a derivation $$S\pro_{k_0,(S,u_0)}v_1\pro_{k_1,(N_1,u_1)}v_2\pro...\pro_{k_{r-1},(N_{r-1},u_{r-1})}v_r$$ is \textbf{leftmost} if each substitution $v_i\pro_{k_i,(N_i,u_i)}v_{i+1}$ acts on the leftmost non-terminal symbol of $v_i$.
\end{mydef}

\begin{mydef}
A context-free grammar $(\cN,\cP,S)$ is called \textbf{unambiguous} if every word in the corresponding context-free language can be obtained by means of a unique leftmost derivation.
\end{mydef}

We point out that, even if two grammars produce the same language, it is possible that one is ambiguous while the other is not. There are also languages which are context-free, but which don't have any unambiguous grammar: such languages are called \textit{inherently ambiguous}.

The reason why we are interested in unambiguous grammars is that it helps controlling the growth of the corresponding language; this will be explained more in detail in Section \ref{SectionComplexityGrowth}.

\begin{myrmk}
We point out that the notion of unambiguous grammar is related to a special kind of push-down automaton, the \textit{deterministic push-down automata}. To be precise, if a language is recognized by a deterministic push-down automaton, then it admits an unambiguous grammar (but the converse is false in general). In Propositions \ref{languagetrivial} and \ref{languagekernel} we will provide two grammars which are unambiguous: the reason behind this is that the corresponding language can be recognized by means of a deterministic push-down automaton.
\end{myrmk}

\subsection{The intersection of a context-free language and a regular language}

Let $\cA$ be an alphabet. Proposition \ref{cfreeintersection} tells us that the intersection of a context-free language $\cL$ and a regular language $\cR$ is a context-free language $\cL\cap\cR$. We provide here an explicit description of a grammar for $\cL\cap\cR$ in terms of a grammar for $\cL$ and of an automaton for $\cR$; we do so in order to provide bounds on the size of this grammar, as well as on the time required to algorithmically produce it; we also give particular attention to the extra property of being unambiguous.

\begin{myprop}\label{cfreeintersection2}
There is an algorithm that, given a context-free grammar $(\cN,\cP,S)$ for $\cL$ and a finite automaton $(Q,\delta,Q_0,Q_f)$ for $\cR$, produces a context-free grammar $(\cN',\cP',S')$ for $\cL\cap\cR$. Moreover the algorithm satisfies the following:

(i) $\norma{\cP'}$ is $O(\norma{\cP}\abs{Q}^{2+2\ram{\cP}})$ and $\ram{\cP'}=\ram{\cP}$.

(ii) The algorithm runs in time $O(\norma{\cP}\abs{Q}^{4+2\ram{\cP}})$.

(iii) If the $(Q,\delta,Q_0,Q_f)$ is deterministic, the algorithm runs in time $O(\norma{\cP}\abs{Q}^{2+2\ram{\cP}})$.

(iv) If $(\cN,\cP,S)$ is unambiguous and $(Q,\delta,Q_0,Q_f)$ is deterministic, then $(\cN',\cP',S')$ is unambiguous.
\end{myprop}

The grammar $(\cN',\cP',S')$ has non-terminal symbols $N_{p\rar q}$ for $N\in\cN$ and $p,q\in Q$. The grammar has an extra symbol $S'$ which we set as initial symbol.
 
Suppose we are given a production rule $(N\pro w_0A_1w_1A_2...A_kw_k)\in\cP$ with $w_0,...,w_k\in\cA^*$ and $A_1,...,A_k\in\cN$; suppose we are given $p,q,p_1,q_1,...,p_k,q_k\in Q$; suppose $\delta(p,w_0)\ni p_1$ and $\delta(q_i,w_i)\ni p_{i+1}$ for $i=1,...,k-1$ and $\delta(q_k,w_k)\ni q$. Then we define a production rule in $\cP'$
$$N_{p\rar q} \ \pro \ w_0(A_1)_{p_1\rar q_1}w_1...(A_k)_{p_k\rar q_k}w_k$$
For each $q_0\in Q_0$ and $q_f\in Q_f$ we define a production rule in $\cP'$
$$S' \ \pro \ S_{q_0\rar q_f}$$
The rest of this section is dedicated to the proof that this grammar has the desired properties.

\begin{mylemma}\label{spezzaparola}
Let $(Q,\delta,Q_0,Q_f)$ be a finite automaton. Let $w=w'w''$ with $w',w''\in\cA^*$ be a word and let $p,q\in Q$ be such that $\delta(p,w)\ni q$. Then there is a state $r\in Q$ such that $\delta(p,w')\ni r$ and $\delta(r,w'')\ni q$.
\end{mylemma}
\begin{proof}
Immediate from the definitions.
\end{proof}

\begin{mylemma}\label{lemma1intersection}
For $w\in\cA^*$ we have that $N_{p\rar q}\der w$ in the grammar $\cP'$ if and only if the following two conditions hold:

(i) $N\der w$ in the grammar $\cP$.

(ii) $\delta(p,w)\ni q$ in the automaton.
\end{mylemma}
\begin{proof}
We prove that if $N\der w$ and $\delta(p,w)\ni q$ then $N_{p\rar q}\der w$, by induction on the length of the derivation $N\der w$. The base step, when the derivation has length one, is trivial. For the inductive step, suppose $N\pro w_0A_1w_1A_2...A_kw_k\der w_0v_1w_1v_2...v_kw_k=w$, meaning that $A_i\der v_i$ for $i=1,...,k$, and suppose $\delta(p,w)=q$. Using repeatedly Lemma \ref{spezzaparola} we find states $p_1,q_1,...,p_k,q_k\in Q$ such that $\delta(p,w_0)\ni p_1$, $\delta(p_i,v_i)\ni q_i$ for $i=1,...,k$, $\delta(q_i,w_i)\ni p_{i+1}$ for $i=1,...,k-1$ and $\delta(q_k,w_k)\ni q$. By definition of the grammar $\cP'$ we have $N_{p\rar q}\pro w_0(A_1)_{p_1\rar q_1}w_1(A_2)_{p_2\rar q_2}...(A_k)_{p_k\rar q_k}w_k$ and by inductive hypothesis we have $(A_i)_{p_i\rar q_i}\der v_i$. It follows that $N_{p\rar q}\der w$ as desired.

We prove that if $N_{p\rar q}\der w$ then $N\der w$ and $\delta(p,w)\ni q$, by induction on the length of the derivation $N_{p\rar q}\der w$. The base step, when the derivation has length one, is trivial. For the inductive step, suppose $N_{p\rar q}\pro w_0(A_1)_{p_1\rar q_1}w_1(A_2)_{p_2\rar q_2}...(A_k)_{p_k\rar q_k}w_k\der w_0v_1w_1v_2...v_kw_k=w$, meaning that $(A_i)_{p_i\rar q_i}\der v_i$ for $i=1,...,k$. By definition of the grammar $\cP'$ we have $N\pro w_0A_1w_1A_2...A_kw_k$ and by inductive hypothesis we have $A_i\der v_i$ for $i=1,...,k$, yielding that $N\der w$. By definition of the grammar $\cP'$ we have that $\delta(p,w_0)\ni p_1$ and $\delta(q_i,w_i)\ni p_{i+1}$ for $i=1,...,k-1$ and $\delta(q_k,w_k)\ni q$. By inductive hypothesis we have that $\delta(p_i,v_i)\ni q_i$ for $i=1,...,k$. It follows that $\delta(p,w)=\delta(p,w_0v_1w_1v_2...v_kw_k)\supseteq\delta(p_1,v_1w_1v_2...v_kw_k)\supseteq\delta(q_1,w_1v_2...v_kw_k)\supseteq...\supseteq\delta(q_k,w_k)\ni q$, and the conclusion follows.
\end{proof}

It immediately follows that the context-free grammar $(\cN',\cP',S')$ produces the language $\cL\cap\cR$. Each production rule in $\cP$ gives us at most $\abs{Q}^{2+2\ram{\cP}}$ production rules in $\cP'$, each of the same length and with the same number of non-terminal symbols; this easily proves that $\norma{\cP'}$ is $O(\norma{\cP}\abs{Q}^{2+2\ram{\cP}})$ and $\ram{\cP'}=\ram{\cP}$.

In order to estimate the complexity of explicitly performing the above construction, we first need to estimate the complexity of the membership problem, for a word of length $l$, to the language generated by the automaton $(Q,\delta,Q_0,Q_f)$. According to Section 4.3.3 of \cite{HMU06}, the membership problem can be solved in time $O(\abs{Q}^2l)$. Now, for every production rule $(N,u)=(N,w_0A_1w_1A_2...A_kw_k)\in\cP$ and for every $p,q,p_1,q_1,...,p_k,q_k\in Q$, we have to test the (equivalent of) membership problem for the words $w_0,...,w_k$: this can be done in time $O(\abs{Q}^2(l(w_0)+...+l(w_k)))$ which is at most $O(\abs{Q}^2(1+l(u)))$; repeating for all choices of $p,q,p_1,q_1,...,p_k,q_k\in Q$ can be done in time $O(\abs{Q}^{4+2k}(1+l(u)))$; repeating for all production rules $(N,u)\in\cP$ can be done in time $O(\abs{Q}^{4+2\ram{\cP}}\norma{\cP})$. In the case the $(Q,\delta,Q_0,Q_f)$ is deterministic, according to Section 4.3.3 of \cite{HMU06}, the membership problem for a word of length $l$ can be solved in time $O(l)$ instead of $O(\abs{Q}^2l)$. The same computation as above gives complexity $O(\abs{Q}^{2+2\ram{\cP}}\norma{\cP})$.

\begin{mylemma}\label{lemma2intersection}
If $(\cN,\cP,S)$ is unambiguous and $(Q,\delta,Q_0,Q_f)$ is deterministic, then $(\cN',\cP',S')$ is unambiguous.
\end{mylemma}
\begin{proof}
Suppose we have a leftmost derivation $N_{p\rar q}\pro...\pro w$ in $(\cN',\cP',S')$. Removing the labels gives a leftmost derivation $N\pro...\pro w$ in $(\cN,\cP,S)$, and thus the productions of this derivation are uniquely determined. Since the automaton is deterministic, the labels $p_i,q_i$ to be added in each of these productions are uniquely determined too. Thus the leftmost derivation $N_{p\rar q}\pro...\pro w$ is uniquely determined too. This shows that $(\cN',\cP',S')$ is unambiguous.
\end{proof}

The proof of Proposition \ref{cfreeintersection2} is thus complete.

\subsection{A grammar for the language of the trivial element}

The result of this section is a particular case of the result of next Section \ref{sec:languagekernel}; we think it's useful to examine the proof in this easier case first.

Let $F_m$ be a free group generated by $b_1,...,b_m$ and consider the alphabet $\cB=\{b_1,...,b_m,\ol{b}_1,...,\ol{b}_m\}$. Let $\cL(1)\subseteq\cB^*$ be the language of all words representing the trivial element $1\in F_m$.

\begin{myprop}\label{languagetrivial}
The language $\cL(1)$ is context-free and has a grammar $(\cN,\cP,S)$ such that:

(i) $\norma{\cP}$ is $O(m^2)$ and $\ram{\cP}=2$.

(ii) $(\cN,\cP,S)$ is unambiguous.
\end{myprop}

The grammar $(\cN,\cP,S)$ has non-terminal symbols $S,B_i,\ol{B}_i$ for $i=1,...,m$. The initial symbol is $S$ and the production rules are as follows:

\vspace{0.2cm}

\begin{tabular}{l l l p{0.5cm} l}
$S$ & $\pro$ & $\epsilon \quad | \quad b_j\ol{B}_jS \quad | \quad \ol{b}_jB_jS$ && for $j=1,...,m$. \\
$B_i$ & $\pro$ & $b_i \quad | \quad \ol{b}_iB_iB_i \quad | \quad b_j\ol{B}_jB_i \quad | \quad \ol{b}_jB_jB_i$ && for $j=1,...,m$ with $j\not=i$. \\
$\ol{B}_i$ & $\pro$ & $\ol{b}_i \quad | \quad b_i\ol{B}_i\ol{B}_i \quad | \quad b_j\ol{B}_j\ol{B}_i \quad | \quad \ol{b}_jB_j\ol{B}_i$ && for $j=1,...,m$ with $j\not=i$. \\
\end{tabular}

\vspace{0.2cm}

The rest of this section is dedicated to the proof that this grammar has the desired properties. When we say that $w'$ is a \textbf{proper} initial segment of $w$ we mean $w'\not=\epsilon$ and $w'\not=w$. The following is a standard lemma about cancellations in words, and the proof will be omitted.

\begin{mylemma}\label{segmentobanale}
Let $w,w_r\in\cB^*$ be non-empty words such that $w\equiv w_r$ in $F_m$ and $w_r$ is reduced. If the first letter of $w$ and the first letter of $w_r$ are different, then there is a proper initial segment $w'$ of $w$ such that $w'\equiv 1$ in $F_m$.
\end{mylemma}

\begin{mylemma}\label{lemma1trivial}
For $w\in\cB^*$ and $i=1,...,m$ we have that $B_i\der w$ (resp. $\ol{B}_i\der w$) if and only if the following two conditions hold:

(i) $w\equiv b_i$ in $F_m$ (resp. $w\equiv\ol{b}_i$).

(ii) No proper initial segment $w'$ of $w$ satisfies $w'\equiv b_i$ in $F_m$ (resp. $w'\equiv\ol{b}_i$).
\end{mylemma}
\begin{proof}
We prove that if $B_i\der w$ then (i) and (ii) hold, by induction on the length of the derivation. The base step, when the derivation has length one, is trivial. Suppose $B_i\der w$ and suppose the first production applied is $B_i\pro\ol{b}_jB_jB_i$ for some $j\in\{1,...,m\}$ (the other case is analogous): then we can write $w=\ol{b}_jxy$ with $B_j\der x$ and $B_i\der y$. By inductive hypothesis $x\equiv b_j$ and $y\equiv b_i$, so that $w\equiv b_i$, yielding (i). Suppose a proper initial segment $w'$ of $w$ satisfies $w'\equiv b_i$: then it is either of the form $w'=\ol{b}_jx'$ where $x'$ is an initial segment of $x$, or of the form $w'=\ol{b}_jxy'$ where $y'$ is a proper initial segment of $y'$. If $w'=\ol{b}_jx'$ then we have $\ol{b}_jx'\equiv b_i$ and thus by Lemma \ref{segmentobanale} we find a proper initial segment $x''$ of $x$ such that $\ol{b}_jx''\equiv 1$, giving $x''\equiv b_j$ and contradicting the inductive hypothesis. If $w'=\ol{b}_jxy'$ then we have $\ol{b}_jxy'\equiv b_i$ and thus $y'\equiv b_i$, contradicting again the inductive hypothesis. This yields (ii).

We prove that if (i) and (ii) holds for $w$ then $B_i\der w$, by induction on the length of $w$. The base step, when $w$ has length one, is trivial. If $w$ has length bigger than one, then by (ii) the first letter of $w$ can't be $b_i$; suppose the first letter of $w$ is $\ol{b}_j$ for some $j\in\{1,...,m\}$ (the other case is analogous). Since $w\equiv b_i$ and the first letter of $w$ is $\ol{b}_j$, by Lemma \ref{segmentobanale} we can find a decomposition $w=\ol{b}_jxy$ with $x\equiv b_j$; without loss of generality we can also assume that $x$ has minimum possible length for such a decomposition. We have that $x\equiv b_j$ and no proper initial segment $x'$ of $x$ satisfies $x'\equiv b_j$, and thus by inductive hypothesis we must have $B_j\der x$. Since $w\equiv b_i$ and $w=\ol{b}_jxy$ we obtain that $y\equiv b_i$; since no proper initial segment $w'$ of $w$ satisfies $w'\equiv b_i$, it follows that no proper initial segment $y'$ of $y$ satisfies $y'\equiv b_i$; it follows by inductive hypothesis that $B_i\der y$. But then we have $B_i\pro\ol{b}_jB_jB_i\der\ol{b}_jxB_i\der\ol{b}_jxy=w$ as desired.
\end{proof}

\begin{mylemma}\label{lemma2trivial}
For $w\in\cB^*$ we have that $S\der w$ if and only if $w\equiv1$ in $F_m$.
\end{mylemma}
\begin{proof}
We prove that if $S\der w$ then $w\equiv1$, by induction on the length of the derivation. The base step, when the derivation has length one, is trivial. Suppose $S\der w$ and suppose the first production applied is $S\pro \ol{b}_iB_iS$ (the other case is analogous): then we can write $w=\ol{b}_ixy$ with $B_i\der x$ and $S\der y$. By Lemma \ref{lemma1trivial} we have $x\equiv b_i$ and by inductive hypothesis we have $y\equiv1$, yielding $w\equiv 1$ as desired.

We prove that if $w\equiv 1$ then $S\der w$, by induction on the length of the word. The base step, when the word has length zero, is trivial. Suppose $w\equiv1$ and without loss of generality assume that $w$ begins with $\ol{b}_j$. Decompose $w=\ol{b}_jxy$ where $x$ is the shortest possible such that $x\equiv b_j$. By Lemma \ref{lemma1trivial} we have $B_j\der x$ and by inductive hypothesis we have $S\der y$, yielding $S\pro\ol{b}_jB_jS\der\ol{b}_jxS\der\ol{b}_jxy=w$ as desired.
\end{proof}

\begin{mylemma}\label{lemma3trivial}
The above grammar is unambiguous.
\end{mylemma}
\begin{proof}
Suppose we are given a leftmost derivation $S=v_0\pro v_1\pro...\pro v_{r-1}\pro v_r=w$ for a certain word $w\in\cB^*$. We fix $k\in\{0,...,r-1\}$ and we want to prove that the production rule to be applied on $v_k$ is forced. Write $v_k=w'Np$ with $w'\in\cB^*$ and $N\in\cN$ and $p\in(\cB\cup\cN)^*$, and notice that $w'$ must be an initial segment of $w$.

Suppose $N=B_i$ for some $i\in\{1,...,m\}$ (the case $N=\ol{B}_i$ is analogous). Each production rule for $B_i$ begins with a terminal symbol, so we must have $w'\not=w$: we call $c\in\cB$ the next letter of $w$, so that $w'c$ is an initial segment of $w$. But for each $c\in\cB^*$ there is exactly one production rule for $B_i$ that begins with $c$, and thus we are forced to apply that production rule.

It is easy to show by induction that each of $v_0,...,v_{r-1}$ contains exactly one occurrence of $S$, and exactly at the end. Suppose $N=S$; if $w'=w$ then we are forced to apply $(S,\epsilon)$; otherwise we have that $w'c$ is an initial segment of $w$ for exactly one $c\in\cB$, and we are forced to apply the unique production rule for $S$ that begins with $c$. This proves that the grammar is unambiguous.
\end{proof}

The proof of Proposition \ref{languagetrivial} is thus complete.

\subsection{A grammar for the language of a kernel}\label{sec:languagekernel}

The following Proposition \ref{languagekernel} is a generalization of Proposition \ref{languagetrivial}, obtained by looking at the proof of Proposition \ref{cfreepreimage}, but with particular care to the size and to the unambiguity of the grammars involved.

Let $F_n$ be a free group generated by $a_1,...,a_n$ and let $\cA=\{a_1,...,a_n,\ol{a}_1,...,\ol{a}_n\}$. Let $F_m$ be a free group generated by $b_1,...,b_m$ and let $\cB=\{b_1,...,b_m,\ol{b}_1,...,\ol{b}_m\}$.

Let $\phi:F_m\rar F_n$ be a homomorphism. For $k=1,...,m$, let $\beta_k,\ol{\beta}_k\in\cA^*$ be the reduced words representing $\phi(b_k),\phi(\ol{b}_k)$ respectively, so that $\ol{\beta}_k$ is the formal inverse of $\beta_k$. Denote with $\norma{\phi}=\max\{l(\beta_1),...,l(\beta_m)\}$.

\begin{myprop}\label{languagekernel}
Let $\cL(\ker(\phi))\subseteq\cB^*$ be the language of all words representing an element of $\ker(\phi)\subseteq F_m$. Then the language $\cL(\ker(\phi))$ is context-free and has a grammar $(\cN,\cP,S)$ such that:

(i) $\norma{\cP}$ is $O(nm^3\norma{\phi}^3)$ and $\ram{\cP}=2$.

(ii) $(\cN,\cP,S)$ is unambiguous.
\end{myprop}

\begin{myrmk}
We can fix $F_m$ and we can take any injective homomorphism $\phi$ from $F_m$ to any finitely generated free group $F_n$. This in particular allows to produce a non-ambiguous grammar $(\cN,\cP,S)$ for the language $\cL(1)\subseteq\cB^*$ with $\ram{\cP}=2$. Thus the above Proposition \ref{languagekernel} gives Proposition \ref{languagetrivial} as a corollary, except for the bound on $\norma{\cP}$ as $O(m^2)$; in that regard, the construction of the proof of Proposition \ref{languagetrivial} is more efficient in this particular case.
\end{myrmk}

Let $\Lambda=\{\lambda\in\cA^* : \lambda$ is a final segment of some of $\beta_1,...,\beta_m,\ol{\beta}_1,...,\ol{\beta}_m\}$, including the empty final segment $\epsilon$ and the full final segments $\beta_1,...,\beta_m,\ol{\beta}_1,...,\ol{\beta}_m$. The grammar $(\cN,\cP,S)$ has non-terminal symbols $S^\lambda,A_i^{\lambda,\mu},\ol{A}_i^{\lambda,\mu}$ for $i=1,...,n$ and $\lambda,\mu\in\Lambda$. The initial symbol is $S=S^\epsilon$ and the production rules are as follows:

\vspace{0.2cm}

\begin{tabular}{l l l p{0.5cm} l}
$S^\epsilon$ & $\pro$ & $\epsilon \quad | \quad b_kS^{\beta_k} \quad | \quad \ol{b}_kS^{\ol{\beta}_k}$ && for $k=1,...,m$. \\
$S^{a_i\lambda}$ & $\pro$ & $\ol{A}_i^{\lambda,\nu}S^\nu$ && for $\nu\in\Lambda$. \\ 
$S^{\ol{a}_i\lambda}$ & $\pro$ & $A_i^{\lambda,\nu}S^\nu$ && for $\nu\in\Lambda$. \\ 
\end{tabular}

\vspace{0.3cm}

\begin{tabular}{l l l p{0.5cm} l}
$A_i^{\epsilon,\mu}$ & $\pro$ & $b_kA_i^{\beta_k,\mu} \quad | \quad \ol{b}_kA_i^{\ol{\beta}_k,\mu}$ && for $k=1,...,m$. \\ 
$A_i^{a_i\lambda,\lambda}$ & $\pro$ & $\epsilon$ && \\ 
$A_i^{a_j\lambda,\mu}$ & $\pro$ & $\ol{A}_j^{\lambda,\nu}A_i^{\nu,\mu}$ && for $j=1,...,n$ with $j\not=i$ and $\nu\in\Lambda$. \\ 
$A_i^{\ol{a}_j\lambda,\mu}$ & $\pro$ & $A_j^{\lambda,\nu}A_i^{\nu,\mu}$ && for $j=1,...,n$ and $\nu\in\Lambda$. \\ 
\end{tabular}

\vspace{0.3cm}

\begin{tabular}{l l l p{0.5cm} l}
$\ol{A}_i^{\epsilon,\mu}$ & $\pro$ & $b_k\ol{A}_i^{\beta_k,\mu} \quad | \quad \ol{b}_k\ol{A}_i^{\ol{\beta}_k,\mu}$ && for $k=1,...,m$. \\ 
$\ol{A}_i^{\ol{a}_i\lambda,\lambda}$ & $\pro$ & $\epsilon$ && \\ 
$\ol{A}_i^{\ol{a}_j\lambda,\mu}$ & $\pro$ & $A_j^{\lambda,\nu}\ol{A}_i^{\nu,\mu}$ && for $j=1,...,n$ with $j\not=i$ and $\nu\in\Lambda$. \\ 
$\ol{A}_i^{a_j\lambda,\mu}$ & $\pro$ & $\ol{A}_j^{\lambda,\nu}\ol{A}_i^{\nu,\mu}$ && for $j=1,...,n$ and $\nu\in\Lambda$. \\ 
\end{tabular}

\vspace{0.2cm}

Notice that there are no production rules for the symbols $A_i^{a_i\lambda,\mu}$ and $\ol{A}_i^{\ol{a}_i\lambda,\mu}$ if $\mu\not=\lambda$. The rest of this section is dedicated to the proof that the above grammar has the desired properties.

Define $\psi:\cB^*\rar\cA^*$ as the unique monoid morphism such that $\psi(b_k)=\beta_k$ and $\psi(\ol{b}_k)=\ol{\beta}_k$ for $k=1,...,m$.

\begin{mylemma}\label{lemma1ker}
For $u\in\cA^*$ and $i=1,...,n$ we have that $A_i^{\lambda,\mu}\der u$ (resp. $\ol{A}_i^{\lambda,\mu}\der u$) if and only if the following three conditions hold:

(i) $\lambda\cdot\psi(u)=\alpha\cdot\mu$ for some $\alpha\in\cA^*$ with $\alpha\equiv a_i$ in $F_n$ (resp. $\alpha\equiv\ol{a}_i$).

(ii) No proper initial segment $\alpha'$ of $\alpha$ satisfies $\alpha'\equiv a_i$ in $F_n$ (resp. $\alpha'\equiv\ol{a}_i$).

(iii) The word $\alpha$ is not an initial segment of $\lambda\cdot\psi(u')$ for any proper initial segment $u'$ of $u$.
\end{mylemma}
\begin{proof}
We prove that if $A_i^{\lambda,\mu}\der u$ then (i), (ii) and (iii) hold, by induction on the length of the derivation. The base step, when the derivation has length one, is trivial.

Suppose $A_i^{\epsilon,\mu}\der u$ and suppose without loss of generality the first production rule used is $A_i^{\epsilon,\mu}\pro b_kA_i^{\beta_k,\mu}$: then we can write $u=b_ku_1$ with $A_i^{\beta_k,\mu}\der u_1$. By inductive hypothesis we have that $\beta_k\cdot\psi(u_1)=\alpha\cdot\mu$ where $\alpha\equiv a_i$ satisfies (ii) and (iii). But then we have $\epsilon\cdot\psi(u)=\alpha\cdot\mu$ where $\alpha\equiv a_i$ satisfies (ii) and (iii), as desired.

Suppose $A_i^{\ol{a}_j\lambda,\mu}\der u$ for some $j\in\{1,...,n\}$: then the first production used is of the form $A_i^{\ol{a}_j\lambda,\mu}\pro A_j^{\lambda,\nu}A_i^{\nu,\mu}$ and we can write $u=u_1u_2$ with $A_j^{\lambda,\nu}\der u_1$ and $A_i^{\nu,\mu}\der u_2$. By inductive hypothesis we can write $\lambda\cdot\psi(u_1)=\alpha_1\cdot\nu$ where $\alpha_1\equiv a_j$ satisfies (ii) and (iii); similarly we can write $\nu\cdot\psi(u_2)=\alpha_2\cdot\mu$ where $\alpha_2\equiv a_i$ satisfies (ii) and (iii). But then we have that $\ol{a}_j\lambda\cdot\psi(u)=\ol{a}_j\alpha_1\alpha_2\cdot\mu$ where $\ol{a}_j\alpha_1\alpha_2\equiv a_i$; property (ii) for $\alpha_1$ and for $\alpha_2$, together with Lemma \ref{segmentobanale}, implies that $\ol{a}_j\alpha_1\alpha_2$ satisfies (ii); property (iii) for $\alpha_1$ and for $\alpha_2$ easily implies property (iii) for $\ol{a}_j\alpha_1\alpha_2$, as desired.

The case $A_i^{a_j\lambda,\mu}\der u$ for some $j\in\{1,...,n\}$ with $j\not=i$ is completely analogous. The case $A_i^{a_i\lambda,\mu}\der u$ can only happen if $\lambda=\mu$ and only in the base step of the induction.

We prove that if $u$ satisfies conditions (i), (ii) and (iii) for some $i\in\{1,...,n\}$ and $\lambda,\mu\in\Lambda$, then $A_i^{\lambda,\mu}\der u$; we proceed by induction on $l(u)+l(\lambda)$. The base step, where $u=\epsilon$ and $\lambda=\epsilon$, is true since (i) can't hold.

Suppose $\lambda=\epsilon$ and $\epsilon\cdot\psi(u)=\alpha\cdot\mu$ where $\alpha\equiv a_i$ satisfies (ii) and (iii). Without loss of generality assume that $u=b_ku_1$ for some $k\in\{1,...,m\}$. We have that $\beta_k\cdot\psi(u_1)=\alpha\cdot\mu$ where $\alpha\equiv a_i$ satisfies (ii) and (iii): by inductive hypothesis we have $A_i^{\beta_k,\mu}\der u_1$. It follows that $A_i^{\epsilon,\mu}\der u$, as desired.

Suppose $\lambda=\ol{a}_j\lambda_1$ for some $j\in\{1,...,n\}$ and suppose $\lambda\cdot\psi(u)=\alpha\cdot\mu$ where $\alpha\equiv a_i$ satisfies (ii) and (iii). Using Lemma \ref{segmentobanale} we can write $\alpha=\ol{a}_j\alpha_1\alpha_2$ where $\alpha_1\equiv a_j$, and without loss of generality we can assume that $\alpha_1$ is the shortest with this property. Write also $u=u_1u_2$ where $u_1$ is the shortest such that $\lambda\cdot\psi(u_1)$ contains $\ol{a}_j\alpha_1$ as an initial segment, and write $\lambda\cdot\psi(u_1)=\ol{a}_j\alpha_1\cdot\nu$ for some $\nu\in\Lambda$. Now we have that $\lambda_1\cdot\psi(u_1)=\alpha_1\cdot\nu$ where $\alpha_1\equiv a_j$ satisfies (ii) and (iii); we also have that $\nu\cdot\psi(u_2)=\alpha_2\cdot\mu$ where $\alpha_2\equiv a_i$ satisfies (ii) and (iii). By inductive hypothesis we have that $A_j^{\lambda_1,\nu}\der u_1$ and $A_i^{\nu,\mu}\der u_2$ and thus $A_i^{\lambda,\mu}\pro A_j^{\lambda_1,\nu}A_i^{\nu,\mu}\der u$, as desired.

The case $\lambda=\ol{a}_j\lambda_1$ for some $j\in\{1,...,n\}$ with $j\not=i$ is analogous.

Suppose $\lambda=a_i\lambda_1$ and suppose $\lambda\cdot\psi(u)=\alpha\cdot\mu$ where $\alpha\equiv a_i$ satisfies (ii) and (iii). Since the first letter of $\lambda$ is $a_i$, the same must hold for $\alpha$ too, and thus property (ii) implies $\alpha=a_i$. Now property (iii) implies $u=\epsilon$, and it follows that $\lambda_1=\mu$. The production rule $A_i^{a_i\lambda_1,\lambda_1}\pro\epsilon$ gives us the conclusion.
\end{proof}

\begin{mylemma}\label{lemma2ker}
For $u\in\cB^*$ we have that $S^\lambda\der u$ if and only if $\lambda\cdot\psi(u)\equiv1$ in $F_n$.
\end{mylemma}
\begin{proof}
We prove that if $S^\lambda\der u$ then $\lambda\cdot\psi(u)\equiv1$, by induction on the length of the derivation. The base step, when the length of the derivation is one, is trivial. Suppose $S^\epsilon\der u$, and without loss of generality write $u=b_ku_1$ for some $k\in\{1,...,m\}$: then the first production applied must be $S^\epsilon\pro b_kS^{\beta_k}\der b_ku_1$, and by inductive hypothesis on $S^{\beta_k}$ we are done. Suppose $S^{a_i\lambda}\der u$: then the first production applied must be of the form $S^{a_i\lambda}\pro\ol{A}_i^{\lambda,\nu}S^\nu$ and the conclusion follows from Lemma \ref{lemma1ker} and from the inductive hypothesis on $S^\nu$.

We prove that if $\lambda\cdot\psi(u)\equiv1$ then $S^\lambda\der u$, by induction on $l(u)+l(\lambda)$. The base step, when $u=\epsilon$ and $\lambda=\epsilon$, is trivial. Suppose $\lambda=\epsilon$ and $\psi(u)\equiv1$, and without loss of generality assume that $u=b_ku_1$ for some $k\in\{1,...,m\}$: then we have $\beta_k\cdot\psi(u_1)\equiv1$ and thus by inductive hypothesis $S^{\beta_k}\der u_1$, which yields $S^\epsilon\pro b_kS^{\beta_k}\der b_ku_1$, as desired.

Finally, suppose $\lambda=\ol{a}_i\lambda_1$ for some $i\in\{1,...,n\}$ (the case $\lambda=a_i\lambda_1$ being analogous) and suppose $\lambda\cdot\psi(u)\equiv1$. Then we have $\lambda_1\cdot\psi(u)\equiv a_i$ and let $\alpha$ be the shortest initial segment of $\lambda_1\cdot\psi(u)$ which satisfies $\alpha\equiv a_i$. Write $u=u_1u_2$ so that $\lambda_1\cdot\psi(u_1)$ contains $\alpha$ as initial segment, and $u_1$ is the shortest with this property. Then we can write $\lambda_1\cdot\psi(u_1)=\alpha\cdot\nu$ for some $\nu\in\Lambda$ and we have that $\alpha\equiv a_i$ and $\alpha$ satisfies properties (ii) and (iii) of Lemma \ref{lemma1ker}: it follows that $A_i^{\lambda_1,\nu}\der u_1$. Since $\nu\cdot\psi(u_2)\equiv1$, by inductive hypothesis we have $S^\nu\der u_2$ and thus $S^\lambda\pro A_i^{\lambda_1,\nu}S^\nu\der u_1u_2$, as desired.
\end{proof}

\begin{mylemma}\label{lemma3ker}
The above grammar is unambiguous.
\end{mylemma}

\begin{proof}
Suppose we are given a leftmost derivation $S^\epsilon=v_0\pro v_1\pro...\pro v_{r-1}\pro v_r=u$ for a certain word $u\in\cB^*$ and we want to prove that the production rule to be applied on $v_k$ is forced, for $k=0,...,r-1$. Write $v_k=u'Np$ with $u'\in\cB^*$ and $N\in\cN$ and $p\in(\cB\cup\cN)^*$, and notice that $u'$ must be an initial segment of $u$, let's say $u=u'u''$.

Suppose $N=A_i^{\epsilon,\mu}$ for some $i\in\{1,...,n\}$ and $\mu\in\Lambda$. Each production for $A_i^{\epsilon,\mu}$ begins with a terminal symbol, and thus $u''$ is non-empty and we can call $d\in\cB$ the first letter of $u''$. But there is exactly one production rule for $A_i^{\epsilon,\mu}$ that begins with $d$, and so we are forced to apply that production rule.

Suppose $N=A_i^{c\lambda,\mu}$ for some $i\in\{1,...,n\}$ and $c\in\cA$ and $c\lambda,\mu\in\Lambda$. If $c=a_i$ then there is at most one derivation for $A_i^{a_i\lambda,\mu}$ and thus we have no choice. If $c=\ol{a}_j$ for some $j\in\{1,...,n\}$, then we only have derivations of the form $A_i^{\ol{a}_j\lambda,\mu}\pro A_j^{\lambda,\nu}A_i^{\nu,\mu}$ and we need to prove that $\nu\in\Lambda$ is uniquely determined. Let's say that during the rest of the derivation the symbol $A_j^{\lambda,\nu}$ derives a word $t\in\cB^*$: we must have that $t$ is an initial segment of $u''$ and thus $\lambda\cdot\psi(t)$ is an initial segment of $\lambda\cdot\psi(u'')$. But we have $\lambda\cdot\psi(t)=\alpha\cdot\nu$ where $\alpha\in\cA^*$ represents $a_j\in F_n$ and properties (ii) and (iii) of Lemma \ref{lemma1ker} hold; property (ii) uniquely determines $\alpha$ to be the shortest initial segment of $\lambda\cdot\psi(u'')$ representing $a_j\in F_n$; property (iii) uniquely determines $t$ to the the shortest initial segment of $u''$ such that $\lambda\cdot\psi(t)$ contains $\alpha$. Since $\lambda\cdot\psi(t)=\alpha\nu$ we have that $\nu$ is uniquely determined too, as desired. The same argument works if $c=a_j$ for some $j\in\{1,...,n\}$ and $j\not=i$.

In a similar way we deal with the case $N=\ol{A}_i^{\lambda,\mu}$ for $i\in\{1,...,n\}$ and $\lambda,\mu\in\Lambda$.

It is easy to show by induction that each of $v_0,...,v_{r-1}$ contains exactly one occurrence of one of $\{S^\lambda : \lambda\in\Lambda\}$, and exactly at the end. Suppose $N=S^\epsilon$: if $u'=u$ then we are forced to use the production rule $S^\epsilon\pro\epsilon$, otherwise we look at the first letter of $u''$ and this forces our choice of production rule to use. Suppose $N=S^{a_i\lambda}$ with $i\in\{1,...,n\}$ and $a_i\lambda\in\Lambda$: then we must use the production rule $S^{a_i\lambda}\pro\ol{A}_i^{\lambda,\nu}S^\nu$ and $\nu\in\Lambda$ is uniquely determined with the same argument as above. The case $N=S^{\ol{a}_i\lambda}$ with $i\in\{1,...,n\}$ and $\ol{a}_i\lambda\in\Lambda$ is completely analogous.

This proves that the grammar is unambiguous.
\end{proof}

From the explicit description of the grammar, we can estimate $\norma{\cP}$. For example, if we count the production rules of the form $A_i^{a_j\lambda,\mu}\ \pro\ \ol{A}_j^{\lambda,\nu}A_i^{\nu,\mu}$ we have $n$ possibilities for the choice of $i$ and $\abs{\Lambda}^3$ possibilities for the choice of the suffixes $a_j\lambda,\mu,\nu$, providing a total of at most $n\abs{\Lambda}^3$ production rules; each of them has length $O(1)$, so this gives a summand at most $O(n\abs{\Lambda}^3)$ in $\norma{\cP}$. The other families of production rules are dealt with in a similar way, yielding that $\norma{\cP}$ is at most $O(nm\abs{\Lambda})+O(n\abs{\Lambda}^3)$. We estimate $\abs{\Lambda}$ as $O(m\norma{\phi})$ and we obtain that $\norma{\cP}$ is $O(nm^3\norma{\phi}^3)$. The proof of Proposition \ref{languagekernel} is thus complete.

\subsection{Complexity of the problem of the existence of equations}\label{SectionComplexity1}

We are ready to provide upper bounds on the running time of the algorithms of Theorems \ref{cfreeIg2} and \ref{cfreeJgd2}. Notice that all the running times provided in this section are polynomial in the input data.

Let $F_n$ be a free group on $n$ generators. Let $H\sgr F_n$ be a finitely generated subgroup of rank $r$, and let $g\in F_n$ be any element. Define $\varphi_g:H*\gen{x}\rar F_n$ be the map that is the inclusion on $H$ and that sends $x$ to $g$; in particular we have $\fI_g=\ker\varphi_g$. Let $L=\norma{\varphi_g}=\max\{l(h_1),...,l(h_r),l(g)\}$ where we mean the length of the reduced words representing those elements.

\begin{mythm}\label{cfreeIg3}
We have the following:

(i) The algorithm in Theorem \ref{cfreeIg2} can be constructed so as to run in time $O(nr^9L^3)$.

(ii) Moreover, this algorithm will output data that is sufficient for a description of an explicit equation (following the procedure described in Proposition \ref{explicit1}).

(iii) If the ideal $\fI_g$ contains a non-trivial equation of degree at most $D$, then the algorithm explicitly writes down a non-trivial equation in time $O(nr^{13}L^7D^2)$.
\end{mythm}
\begin{proof}
According to Proposition \ref{languagekernel}, we have an (unambiguous) grammar $(\cN,\cP,S)$ for $\cL(\fI_g)$ such that $\norma{\cP}$ is $O(nr^3L^3)$ and $\ram{\cP}=2$. For the regular language $\cR$ of all non-empty reduced words in $\{h_1,...,h_r,x,\ol{h}_1,...,\ol{h}_r,\ol{x}\}$, it is easy to provide a deterministic finite automaton $(Q,\delta,\{q_0\},Q_f)$ such that $\abs{Q}$ is $O(r)$. Thus we can produce a context-free grammar $(\cN',\cP',S')$ for the (unambiguous) language $\cL(\fI_g)\cap\cR$ such that $\norma{\cP'}$ is $O(nr^9L^3)$. The first statement of Theorem \ref{cfreeIg3} now follows from Propositions \ref{nonempty2} and \ref{explicit1}.

If the ideal contains an equation of degree at most $D$, then by \cite{PART1} we have that $\fI_g$ must contain a non-trivial equation of length $O(L_0^2D)$, where $L_0$ is the number of edges in the core graph $\bcore{H}\vee\bcore{\gen{g}}$ as defined in \cite{PART1}. Since the number of edges in $\bcore{H}$ is at most $l(h_1)+...+l(h_r)$ and the number of edges in $\bcore{\gen{g}}$ is at most $l(g)$, we have that $L_0\le l(h_1)+...+l(h_r)+l(g)$ and thus $L_0$ is $O(rL)$. Thus $\fI_g$ must contain a non-trivial equation of length $O(r^2L^2D)$, and it follows from Proposition \ref{explicit3} that the algorithm writes down an explicit equation in time $O(nr^{13}L^7D^2)$.
\end{proof}

\begin{mythm}\label{cfreeJgd3}
We have the following:

(i) The algorithm in Theorem \ref{cfreeJgd2} can be constructed so as to run in time $O(nr^{15}L^3d^6)$.

(ii) Moreover, this algorithm will output data sufficient for a description of an explicit equation (following the procedure described in Proposition \ref{nonempty2}).

(iii) In addition, this algorithm explicitly writes down an equation in $\fJ_{g,d}$ in time $O(nr^{23}L^{11}d^{10})$.
\end{mythm}
\begin{proof}
Completely analogous to the proof of Theorem \ref{cfreeIg3}. Instead of the results of Section 4 of \cite{PART1} we use the results of Section 5.
\end{proof}

\subsection{The minimum possible degree for equations in an ideal}

\begin{mythm}\label{dmin}
There is an algorithm that, given $H\sgr F_n$ finitely generated and $g\in F_n$, computes the minimum possible degree for an equation in $\fI_g$. The algorithm runs in time $O(nr^9L^3\log(nrL))$.
\end{mythm}
\begin{proof}
We take the unambiguous grammar for $\cL(\fI_g)$ as described in Proposition \ref{languagekernel} and we run the algorithm of Proposition \ref{explicit2}; as size function we use $\sigma:\{h_1,...,h_r,x,\ol{h}_1,...,\ol{h}_r,\ol{x}\}\rar\bN$ given by $\sigma(h_i)=\sigma(\ol{h}_i)=0$ for $i=1,...,r$ and $\sigma(x)=\sigma(\ol{x})=1$.
\end{proof}

We now discuss the problem of finding an estimate on the minimum possible degree $d_{\min}$ for a non-trivial equation in $\fI_g$. A good bound on $d_{\min}$ leads to a good bound on the running time of the algorithm of Theorem \ref{cfreeIg3}. We first need a preliminary proposition.

\begin{myprop}\label{inversenorm}
Let $F_m$ be a finitely generated free group with a fixed basis $c_1,...,c_m$. Let $\psi:F_m\rar F_m$ be an automorphism and suppose each of $\psi(c_1),...,\psi(c_m)$ has length at most $C$ when written in the basis $c_1,...,c_m$. Then there is an inner automorphism $\rho:F_m\rar F_m$ such that each of $\psi^{-1}\circ\rho(c_1),...,\psi^{-1}\circ\rho(c_m)$ has length at most $K_mC^{M_m}$ when written in the basis $c_1,...,c_m$, for some constants $K_m,M_m$ independent on $\psi$.
\end{myprop}
\begin{proof}
Immediately follows from Corollary 4.6 of \cite{LSV15}. The constants $K_m,M_m$ are the ones introduced in \cite{LSV15} too.
\end{proof}

Mantaining the notation of Section \ref{SectionComplexity1}, fix $F_n$ and $H\sgr F_n$ finitely generated and $g\in F_n$; denote with $h_1,...,h_r$ a fixed basis for $H$, with $\varphi_g:H*\gen{x}\rar F_n$ the evaluation map (whose kernel is $\fI_g$), and with $L=\norma{\varphi_g}=\max\{l(h_1),...,l(h_r),l(g)\}$.

\begin{mythm}\label{bounddmin}
If it is non-trivial, the ideal $\fI_g$ contains a non-trivial equation of degree at most $2K_{r+1}L^{M_{r+1}}$.
\end{mythm}
\begin{proof}
As explained in $\Psecidealfingen$ of \cite{PART1}, we can find a basis $c_1,...,c_{r+1}$ for $H*\gen{x}$ with the following properties:

(i) $\fI_g$ is generated by words of length at most $2$ in $c_1,...,c_{r+1}$;

(ii) The length of $h_1$ (resp. $h_2,...,h_r,x$) written as a reduced word in $c_1,...,c_{r+1}$ is smaller or equal than the length of $h_1$ (resp. $h_2,...,h_r,g$) written as a reduced word in $a_1,...,a_n$.

The basis $c_1,...,c_{r+1}$ is obtained by looking at certain paths in the folded graph at the end of a suitable sequence of Stallings' folding operations; the details of the construction can be found in the remark at the end of Section 3 of \cite{PART1}. In particular, each of $h_1,...,h_r,x$ has length at most $L$ when written as a reduced word in $c_1,...,c_{r+1}$.

Take the automorphism $\psi:H*\gen{x}\rar H*\gen{x}$ with $\psi(c_1)=h_1,...,\psi(c_r)=h_r,\psi(c_{r+1})=x$. Each of $\psi(c_1),...,\psi(c_{r+1})$ has length at most $L$ when written in the basis $c_1,...,c_{r+1}$. By Proposition \ref{inversenorm}, there is a conjugation automorphism $\rho:H*\gen{x}\rar H*\gen{x}$ such that each of $\psi^{-1}\circ\rho(c_1),...,\psi^{-1}\circ\rho(c_{r+1})$ has length at most $K_{r+1}L^{M_{r+1}}$ when written in the basis $c_1,...,c_{r+1}$.

In particular we have that each of $\psi(\psi^{-1}\circ\rho(c_1)),...,\psi(\psi^{-1}\circ\rho(c_{r+1}))$ has length at most $K_{r+1}L^{M_{r+1}}$ when written in the basis $\psi(c_1),...,\psi(c_{r+1})$; this means that each of $\rho(c_1),...,\rho(c_{r+1})$ has length at most $K_{r+1}L^{M_{r+1}}$ when written in the basis $h_1,...,h_r,x$. We can take an equation which is a word of length at most two in $\rho(c_1),...,\rho(c_{r+1})$, and thus is a word of length at most $2K_{r+1}L^{M_{r+1}}$ in $h_1,...,h_r,x$. In particular its degree is at most $2K_{r+1}L^{M_{r+1}}$, as desired.
\end{proof}

One may hope to improve the bound of Theorem \ref{bounddmin} to a bound which is polynomial in both $r,L$. Unfortunately this isn't possible, as shown in Section \ref{exampledmin}, where we provide examples of ideals $\fI_g$ where $r=n+1$ and $L\approx p^2$ and the minimum possible degree for a non-trivial equation is at least $\approx p^n$, which is $\approx L^{\frac{r-1}{2}}$.

\subsection{Complexity of the problem of computing the growth rates}\label{SectionComplexityGrowth}

\begin{mythm}\label{growth3}
There is an algorithm as in Theorem \ref{growth2} that runs in polynomial time.
\end{mythm}
\begin{proof}
According to Proposition \ref{languagekernel} we can build, in polynomial time in $H,g,d$, a grammar $(\cP,\cN,S)$ for $\cL(\fJ_{g,d})$ with polynomial size $\norma{\cP}$. The algorithm of Theorem \ref{cfreegrowth} runs in polynomial time too, according to \cite{GKRS08}. We run such algorithm on such grammar, and the conclusion follows.
\end{proof}

We point out that most of the complexity of the algorithm of Theorem \ref{cfreegrowth} comes from the possible presence of ambiguity in the grammar. If we are only interested in unambiguous grammars, then the algorithm can be made much easier and much faster. Since our grammar is obtained using Propositions \ref{languagekernel} and \ref{cfreeintersection2}, it is unambiguous, and so we are able to apply the easier version of the algorithm.

\section{Examples}\label{SectionExamples}

We now provide a few examples. In each of them we have a free group $F_n$, a subgroup $H\sgr F_n$ and an element $g\in F_n$ such that $\fI_g$ is non-trivial; we are interested in studying the growth rate of the functions $\rho_{g,d}$ according to Theorem \ref{growth2}. Some of the examples are taken from \cite{PART1}.

In Section \ref{examplecyclic} we deal with the case where $\rank{H}=1$. We show that $\rho_{g,d}(M)$ has polynomial growth rate for every $d\in D_g$, in contrast with Theorem \ref{Dgpolk}.

In Example \ref{example46} we show that it is possible that the function $\rho_{g,d}(M)$ has polynomial growth rate of degree zero, i.e. it is constant for all $M$ big enough.

We also provide the additional Example \ref{exampledmin}, where the minimum degree in $D_g$ is big compared to the size of the subgroup $H$.

\subsection{Cyclic subgroups}\label{examplecyclic}

We here deal with the case where $\rank{H}=1$. As shown in $\Pexamplecyclic$ of \cite{PART1}, we can assume that $F_n=\gen{a}$ and that $H=\gen{h}$ where $h=a^m$ with $m\ge1$, and that $g=a^k$ with $k\ge0$ coprime with $m$.

Every cyclically reduced equation in $\fI_g$ has the form $w=h^{e_0}x^{f_1}h^{e_1}...h^{e_{s-1}}x^{f_s}h^{e_s}$ with $s\ge1$ and $e_1,...,e_{s-1},f_1,...,f_s\in\bZ\setminus\{0\}$ and $e_0,e_s\in\bZ$ with the same sign (possibly one or both of them equal to zero). When we substitute $h=a^m$ and $x=a^k$ we obtain the trivial word, and thus we must have $(e_0+...+e_s)m+(f_1+...+f_s)k=0$. The degree of such equation is $\abs{f_1}+...+\abs{f_s}$.

We want to estimate $\rho_{g,d}(M)$ and thus we restrict our attention to equations of degree $d$ and length at most $M$. This implies that $s\le d$ and gives only a finite number of possible choices for the numbers $f_1,...,f_s$. Once $s,f_1,...,f_s$ are fixed, we must choose $e_0,...,e_s$ such that $e_0+...+e_s=-(f_1+...+f_s)k/m$ and such that $\abs{e_0}+...+\abs{e_s}\le M-d$. The number of possible choices is polynomial of degree at most $s$ in $M$.

This shows that $\rho_{g,d}(M)$ has polynomial growth of degree at most $d$ for every $m\ge1$ and $k\ge0$ and $d\in D_g$.

\subsection{An ideal with a finite number of minimum-degree equations}\label{example46}

As in $\Pexampleeven$ of \cite{PART1}, let $F_2=\gen{a,b}$ and consider the subgroup $H=\gen{h_1,h_2}$ with $h_1=ba$ and $h_2=ab^2\ol{a}$ and the element $g=a$. In this case the ideal $\fI_g$ is the normal subgroup generated by the equation $\ol{x}h_2xx\ol{h}_1x\ol{h}_1$ and we have $D_g=\{d : d\ge4$ even$\}$.

The function $\rho_{g,4}(M)$ is constant for all $M$ big enough, since $\fI_g$ only contains one equation of degree $4$ up to inverse and conjugations, which is the generator of the ideal. The function $\rho_{g,6}(M)$ has polynomial growth of degree at least one. The function $\rho_{g,d}(M)$ has exponential growth for every $d\ge8$ even, by Lemma \ref{expdegrees}.

\subsection{An ideal where the minimum-degree equations have linear growth}\label{example23}

As in $\Pexampleodd$ of \cite{PART1}, let $F_2=\gen{a,b}$ and consider the subgroup $H=\gen{h_1,h_2}$ with $h_1=b$ and $h_2=ababa$. Let $g=a$. The ideal $\fI_g$ is the normal subgroup generated by $\ol{h}_2xh_1xh_1x$. We have $D_g=\{d : d\ge2$ integer$\}$ as the ideal contains equations of degree $2$ (for example $h_2h_1x\ol{h}_2\ol{h}_1\ol{x}\in\fI_g$).

The function $\rho_{g,2}(M)$ has linear growth, since a complete list of the degree $2$ equations, up to inverses and conjugations, is given by $[(h_2h_1)^i,xh_1]$ for $i\in\bZ\setminus\{0\}$. The function $\rho_{g,3}(M)$ has polynomial growth of degree at least one. The function $\rho_{g,d}(M)$ has exponential growth for $d\ge4$, by Lemma \ref{expdegrees}.

\subsection{An ideal with no small-degree equations}\label{exampledmin}

The idea of this example is taken from Section 4.1 of \cite{LSV15}, but some changes were needed to fit our purposes. Fix a natural number $p$ and set $u(y,x)=xyx^2yx^3...yx^{p+1}$.

For $n\ge2$ let $F_{n+1}=\gen{a_1,...,a_n,b}$. Consider the subgroup $H=\gen{h_1,...,h_n,h'}$ where $h_1=a_1$ and $h_i=a_i\cdot(u(a_{i-1},b))^{-1}$ for $i=2,...,n$ and $h'=\ol{a}_n\ol{b}$. Consider the element $g=b$.

Consider the alphabet $\cH=\{h_1,\ol{h}_1,...,h_n,\ol{h}_n,h',\ol{h'},x,\ol{x}\}$. Define $\oc{a}_1,...,\oc{a}_n\in\cH^*$ as follows: we set $\oc{a}_1=h_1$, and we define inductively $\oc{a}_i=h_iu(\oc{a}_{i-1},x)$. Consider the natural projection map $\eta:\cH^*\rar H*\gen{x}$ defined as at the beginning of Section \ref{Subsectioneta}; consider also the evaluation homomorphism $\varphi_g:H*\gen{x}\rar F_{n+1}$ which is the identity on $H$ and with $\varphi_g(x)=g$. It is easy to show by induction that $\varphi_g(\eta(\oc{a}_i))=a_i$. It is also easy to prove by induction that

(a) The word $\oc{a}_i$ only contains the letters $h_1,...,h_i,x$.

(b) For every two consecutive letters in $\oc{a}_i$, at least one is $x$.

(c) The word $\oc{a}_i$ has length $\frac{p^2+3p+4}{2}\cdot\frac{p^{i-1}-1}{p-1}+p^{i-1}$.

Finally, define the element $w\in\cH^*$ given by $w=h'x\oc{a}_n$. In figure \ref{figexampledmin} we can see the core graph $\bcore{H}$, as defined in \cite{PART1}. By means of the algorithm described in Section 3 of \cite{PART1}, we can prove that $\fI_g$ is generated, as a normal subgroup of $H*\gen{x}$, by the single element $\eta(w)$. 

We now want to apply Theorem 4.4 of \cite{LS01} to the ideal $\fI_g$. We take the symmetrized set $W$ given by the cyclic permutations of $w$ and of $\ol{w}$, and we want to prove that the small cancellation property $C'(1/6)$ holds. A cyclic permutation of $w$ and a cyclic permutation of $\ol{w}$ can't have an initial segment in common, since $w$ only contains the letters $h_1,...,h_n,h',x$. Suppose $w',w''$ are distinct cyclic permutations of $w$ with a common initial segment $u$. We observe that, since $w$ only contains one occurrence of $h'$, $u$ can't contain $h'$. We can write
$$w=h'xh_n\cdot xh_{n-1}u(\oc{a}_{n-2},x)\cdot x^2h_{n-1}u(\oc{a}_{n-2},x)\cdot ...\cdot x^ph_{n-1}u(\oc{a}_{n-2},x)\cdot x^{p+1}$$
and it is easy to see that, between two consecutive occurrences of $h_{n-1}$, there is always a different number of letters. This implies that $u$ contains at most one occurrence of $h_{n-1}$, and in particular $l(u)$ is at most $\approx\frac{2}{p}\cdot l(w)$. For $p$ big enough this proves that $C'(1/6)$ holds in this case.

By Dehn's algorithm (Theorem 4.4 of \cite{LS01}), every equation in $\fI_g$ has a subword of length $\frac{1}{2}\abs{w}$ in common with a cyclic permutation of $w$ or $\ol{w}$. For any two consecutive letters of $w$, at least one is $x$; a subword of length $\frac{1}{2}\abs{w}$ has to contain at least $\frac{1}{4}\abs{w}-1$ letters $x$ or $\ol{x}$. This shows that every non-trivial equation in $\fI_g$ has degree at least
$$\frac{p^2+3p+4}{8}\cdot\frac{p^{n-1}-1}{p-1}+\frac{1}{4}p^{n-1}-\frac{1}{2}$$
which, for $p$ big enough, is at least $p^n/8$.

Thus in this case we have a subgroup $H$ with $n+1$ generators, each of length at most $\frac{p^2+5p+6}{2}$, and where the minimum degree for a non-trivial equation in $\fI_g$ is at least $p^n/8$.

\begin{figure}[h!]
\centering
\begin{tikzpicture}

\node (1) at (6,0) {$*$};
\node (2) at (-6,0) {$.$};

\draw[->,out=120,in=60,looseness=17] (1) to node[above]{$a_1$} (1);
\draw[->] (1) to node[above]{$b$} (2);
\draw[->,out=30,in=150,looseness=1] (2) to node[above]{$a_n$} (1);

\node (3a) at (4,-1) {$.$};
\node (3b) at (-5,-0.5) {$.$};
\node (3c) at (-4,-1) {$.$};
\node (3d) at (-3,-1) {$.$};
\node (3e) at (-2,-1) {$.$};
\draw[->] (1) to node[left]{$a_2$} (3a);
\draw[->] (2) to node[right]{$a_1$} (3b);
\draw[->] (3b) to node[above]{$b$} (3c);
\draw[->] (3c) to node[above]{$b$} (3d);
\draw[->] (3d) to node[above]{$a_1$} (3e);
\draw[dotted] (3e) to (3a);

\node (4a) at (4,-2) {$.$};
\node (4b) at (-5,-1) {$.$};
\node (4c) at (-4,-2) {$.$};
\node (4d) at (-3,-2) {$.$};
\node (4e) at (-2,-2) {$.$};
\draw[->] (1) to node[left]{$a_3$} (4a);
\draw[->] (2) to node[below]{$a_2$} (4b);
\draw[->] (4b) to node[above]{$b$} (4c);
\draw[->] (4c) to node[above]{$b$} (4d);
\draw[->] (4d) to node[above]{$a_2$} (4e);
\draw[dotted] (4e) to (4a);

\node (5a) at (4,-5) {$.$};
\node (5b) at (-5,-2.5) {$.$};
\node (5c) at (-4,-5) {$.$};
\node (5d) at (-3,-5) {$.$};
\node (5e) at (-2,-5) {$.$};
\draw[->] (1) to node[left]{$a_n$} (5a);
\draw[->] (2) to node[left]{$a_{n-1}$} (5b);
\draw[->] (5b) to node[right]{$b$} (5c);
\draw[->] (5c) to node[above]{$b$} (5d);
\draw[->] (5d) to node[above]{$a_{n-1}$} (5e);
\draw[dotted] (5e) to (5a);

\draw[loosely dotted] (0,-2.5) to (0,-4.5);

\end{tikzpicture}
\caption{The core graph $\core{H}$, as defined in \cite{PART1}, of the subgroup $H$ of Section \ref{exampledmin}.}\label{figexampledmin}
\end{figure}
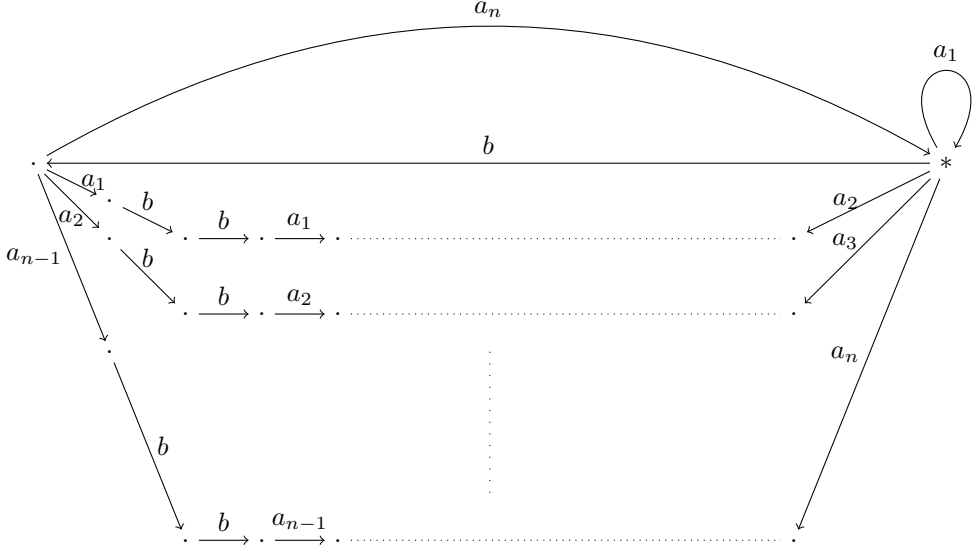

\section{Equations in more variables}\label{SectionMultivariate}

We point out that most of the results of this paper can be generalized to equations in more than one variable. We gather the statements of these results in this section, but since the proofs are completely analogous to the one-variable case, we only provide sketches. Let $F_n$ be a free group generated by $n$ elements $a_1,...,a_n$.

Let $H\sgr F_n$ be a finitely generated subgroup and let $\gen{x_1},\gen{x_2},...,\gen{x_m}\cong\bZ$ be infinite cyclic groups with generators $x_1,...,x_m$ respectively. Let $(g_1,...,g_m)\in (F_n)^m$ be an $m$-tuple of elements; see \cite{PART1} for the definition of the ideal $\fI_{g_1,...,g_m}$, as well as for the definition of multi-degree $(d_1,...,d_m)$ of an equation $w\in\fI_{g_1,...,g_m}$. Theorems \ref{cfreeIg}, \ref{cfreeIg2} and \ref{cfreeIg3} can be generalized as follows:

\begin{mythm}\label{theorem7.1}
The set $\fI_{g_1,...,g_m}$ is context-free as a subset of $H*\gen{x_1}*...*\gen{x_m}$. There is an algorithm that, given $H,g_1,...,g_m$, tells us whether $\fI_{g_1,...,g_m}$ contains a non-trivial equation and, in case it does, also produces data sufficient for an explicit description of a non-trivial equation in $\fI_{g_1,...,g_m}$ (as specified in Proposition \ref{explicit1}). The algorithm runs in polynomial time.
\end{mythm}
\begin{proof}[Sketch of proof]
The set $\fI_{g_1,...,g_m}$ can be seen as the kernel of a homomorphism between free groups, and thus is context-free. A context-free grammar $(\cN,\cP,S)$ for $\cL(\fI_{g_1,...,g_m})$ can be computed algorithmically, according to Proposition \ref{languagekernel}; the grammar has $\norma{\cP}$ polynomially bounded, and can be computed in polynomial time. The conclusion follows from Propositions \ref{nonempty2} and \ref{explicit1}.
\end{proof}

As in the one-variable case, we consider the subset $\fJ_{g_1,...,g_m,(d_1,...,d_m)}$ of $\fI_{g_1,...,g_m}$ given by all the equations of multi-degree $(d_1,...,d_m)$. Theorems \ref{cfreeJgd}, \ref{cfreeJgd2} and \ref{cfreeJgd3} can be generalized as follows:

\begin{mythm}
The set $\fJ_{g_1,...,g_m,(d_1,...,d_m)}$ is context-free as a subset of $H*\gen{x_1}*...*\gen{x_m}$. There is an algorithm that, given $H,g_1,...,g_m,d_1,...,d_m$, tells us whether the set $\fJ_{g_1,...,g_m,(d_1,...,d_m)}$ is non-empty and, if so, explicitly writes down an element of $\fJ_{g_1,...,g_m,(d_1,...,d_m)}$. The algorithm runs in polynomial time.
\end{mythm}
\begin{proof}[Sketch of proof]
We intersect the context-free language $\cL(\fI_{g_1,...,g_m})$ with the regular language of all cyclically reduced words that contain exactly $d_i$ occurrences of $x_i$ and $\ol{x}_i$, for $i=1,...,m$; this proves that $\clan{\fJ_{g_1,...,g_m,(d_1,...,d_m)}}$ is context-free. We take a grammar $(\cN,\cP,S)$ for $\clan{\fJ_{g_1,...,g_m,(d_1,...,d_m)}}$ and we add production rules as in the proof of Theorem \ref{cfreeJgd}, in order to obtain a bigger context-free language $\cL'$ such that $\rlan{\fJ_{g_1,...,g_m,(d_1,...,d_m)}}\subseteq\cL'\subseteq\cL(\fJ_{g_1,...,g_m,(d_1,...,d_m)})$. We intersect $\cL'$ with the regular language of all reduced words, and we obtain that $\rlan{\fJ_{g_1,...,g_m,(d_1,...,d_m)}}$ is context-free, proving the first sentence of the statement.

As in \ref{theorem7.1}, we can produce in polynomial time a context-free grammar $(\cM,\cQ,T)$ for $\cL(\fI_{g_1,...,g_m})$ with $\norma{\cQ}$ polynomially bounded. By Proposition \ref{cfreeintersection2} we can produce in polynomial time a context-free grammar $(\cM',\cQ',T')$ for $\clan{\fJ_{g_1,...,g_m,(d_1,...,d_m)}}$ with $\norma{\cQ'}$ polynomially bounded. We now run the algorithm of Proposition \ref{nonempty2} on this grammar: we can check in polynomial time whether the language is non-empty. In order to produce in polynomial time an explicit equation, we use the results of \cite{PART1} to obtain a polynomial bound on the length of such an equation, and then we apply the algorithm of Proposition \ref{explicit3}.
\end{proof}

Let $L=\max\{l(h_1),...,l(h_r),l(g_1),...,l(g_m)\}$, where $l(h_1),...,l(h_r),l(g_1),...,l(g_m)$ are the lengths of the reduced words in the basis $a_1,...,a_n$ representing those elements. It is possible to generalize Theorem \ref{bounddmin} as follows:

\begin{mythm}
The ideal $\fI_{g_1,...,g_m}$ contains a non-trivial equation of multi-degree $(d_1,...,d_m)$ with $d_1+...+d_m\le 2K_{r+m}L^{M_{r+m}}$.
\end{mythm}
\begin{proof}
As explained in $\Psecidealfingen$ of \cite{PART1}, we can find a basis $c_1,...,c_{r+m}$ for $H*\gen{x_1}*...*\gen{x_m}$ with the following properties:

(i) $\fI_g$ is generated by words of length at most $2$ in $c_1,...,c_{r+m}$;

(ii) The length of $h_1$ (resp. $h_2,...,h_r,x_1,...,x_m$) written as a reduced word in $c_1,...,c_{r+m}$ is smaller or equal than the length of $h_1$ (resp. $h_2,...,h_r,g_1,...,g_m$) written as a reduced word in $a_1,...,a_n$.

In particular, each of $h_1,...,h_r,x_1,...,x_m$ has length at most $L$ when written as a reduced word in $c_1,...,c_{r+m}$.

We now take the automorphism $\psi:H*\gen{x_1}*...*\gen{x_m}\rar H*\gen{x_1}*...*\gen{x_m}$ with $\psi(c_1)=h_1,...,\psi(c_r)=h_r,\psi(c_{r+1})=x_1,...,\psi(c_{r+m})=x_m$. Each of $\psi(c_1),...,\psi(c_{r+m})$ has length at most $L$ when written in the basis $c_1,...,c_{r+m}$. By Proposition \ref{inversenorm}, there is a conjugation automorphism $\rho:H*\gen{x_1}*...*\gen{x_m}\rar H*\gen{x_1}*...*\gen{x_m}$ such that each of $\psi^{-1}\circ\rho(c_1),...,\psi^{-1}\circ\rho(c_{r+m})$ has length at most $K_{r+m}L^{M_{r+m}}$ when written in the basis $c_1,...,c_{r+m}$.

In particular we have that each of $\psi(\psi^{-1}\circ\rho(c_1)),...,\psi(\psi^{-1}\circ\rho(c_{r+m}))$ has length at most $K_{r+m}L^{M_{r+m}}$ when written in the basis $\psi(c_1),...,\psi(c_{r+m})$; this means that each of $\rho(c_1),...,\rho(c_{r+m})$ has length at most $K_{r+m}L^{M_{r+m}}$ when written in the basis $h_1,...,h_r,x_1,...,x_m$. We can take an equation $w$ which is a word of length at most $2$ in $\rho(c_1),...,\rho(c_{r+m})$, and thus is a word of length at most $2K_{r+m}L^{M_{r+m}}$ in $h_1,...,h_r,x_1,...,x_m$. In particular, the total number of occurrences of $x_1,\ol{x}_1,...,x_m,\ol{x}_m$ in $w$ is at most $2K_{r+m}L^{M_{r+m}}$, as desired.
\end{proof}

As in Section \ref{SubsectionGrowth}, we can define the functions $\rho_{g_1,...,g_m}(M)$ and $\rho_{g_1,...,g_m,(d_1,...,d_m)}(M)$ as the growth functions of the languages $\clan{\fI_{g_1,...,g_m}}$ and $\clan{\fJ_{g_1,...,g_m,(d_1,...,d_m)}}$ respectively. The analogous of Theorems \ref{growth1}, \ref{growth2} and \ref{growth3} hold:

\begin{mythm}
The function $\rho_{g_1,...,g_m}(M)$ has exponential growth.
\end{mythm}
\begin{proof}[Sketch of proof]
By Theorem \ref{cfreegrowth} we have that $\clan{\fI_{g_1,...,g_m}}$ has either polynomial or exponential growth. We deal with two separate cases, depending on whether $H$ has rank $1$ or at least $2$. In both cases, the exact same construction as in the proof of Theorem \ref{growth1} provides an exponential lower bound, providing the desired result.
\end{proof}

\begin{mythm}
The function $\rho_{g_1,...,g_m,(d_1,...,d_m)}(M)$ has either exponential growth, or polynomial growth of degree $k$ for some $k\in\bN$. Moreover, there is an algorithm that, in polynomial time, tells us which case takes place, and in the second case computes the degree $k$ of the growth.
\end{mythm}
\begin{proof}
By Theorem \ref{languagekernel} we can produce in polynomial time a grammar $(\cN,\cP,S)$ for the language $\cL(\fI_{g_1,...,g_m})$, with $\norma{\cP}$ polynomially bounded. By Proposition \ref{cfreeintersection2} we can intersect with the regular language of all cyclically reduced words containing $d_i$ occurrences of $x_i$ and $\ol{x}_i$, for $i=1,...,m$; this produces in polynomial time a grammar $(\cN',\cP',S')$ for the language $\clan{\fJ_{g_1,...,g_m,(d_1,...,d_m)}}$, with $\norma{\cP'}$ polynomially bounded. By Theorem \ref{cfreegrowth} we have that $\clan{\fI_{g_1,...,g_m,(d_1,...,d_m)}}$ has either exponential growth or polynomial growth of degree $k$ for some $k\in\bN$, and the type of growth (including the degree) can be computed in polynomial time.
\end{proof}

\bibliographystyle{alpha}
\nocite{*}
\bibliography{bibliography.bib}

\end{document}